\documentclass[11pt,reqno]{amsart}
\usepackage{lineno,hyperref}
\usepackage{lineno,hyperref}
\modulolinenumbers[5]
\usepackage{amsmath}
\usepackage{amssymb}
\usepackage{mathrsfs}
\usepackage{amsthm}
\usepackage{bm}
\usepackage{graphicx}
\usepackage{booktabs}
\usepackage{float}
\usepackage{subfig}
\usepackage{geometry}
\usepackage{color}

\newtheorem{Def}{Definition}[section]
\newtheorem{lem}[Def]{Lemma}
\newtheorem{theo}[Def]{Theorem}
\newtheorem{pro}[Def]{Proposition}
\newtheorem{rem}[Def]{Remark}

\newcommand{\lam}{\lambda_2-\lambda_1}
\newcommand{\mcal}{\mathcal}
\newcommand{\mscr}{\mathscr}
\newcommand{\mbb}{\mathbb}
\newcommand{\mbf}{\mathbf}

\newcommand{\ud}{\mathrm d}

\numberwithin{equation}{section}
\allowdisplaybreaks[4]

\begin{document}

\title[LDP]{Numerically asymptotical preservation of the 
large deviations principles for  invariant measures of  Langevin equations}

\author{Jialin Hong}
\address{Academy of Mathematics and Systems Science, Chinese Academy of Sciences, Beijing
	100190, China; School of Mathematical Sciences, University of Chinese Academy of
	Sciences, Beijing 100049, China}
\email{hjl@lsec.cc.ac.cn}

\author{Diancong Jin}
\address{Academy of Mathematics and Systems Science, Chinese Academy of Sciences, Beijing
	100190, China; School of Mathematical Sciences, University of Chinese Academy of
	Sciences, Beijing 100049, China}
\email{diancongjin@lsec.cc.ac.cn (Corresponding author)}

\author{Derui Sheng}
\address{Academy of Mathematics and Systems Science, Chinese Academy of Sciences, Beijing
	100190, China; School of Mathematical Sciences, University of Chinese Academy of
	Sciences, Beijing 100049, China}
\email{sdr@lsec.cc.ac.cn}

\author{Liying Sun}
\address{Academy of Mathematics and Systems Science, Chinese Academy of Sciences, Beijing
	100190, China; School of Mathematical Sciences, University of Chinese Academy of
	Sciences, Beijing 100049, China}
\email{liyingsun@lsec.cc.ac.cn}

\thanks{This work is supported by National Natural Science Foundation of China (Nos.  11971470, 11871068, 11926417, 12031020 and 12022118).}


\begin{abstract}
	In this paper, we focus on two kinds of large deviations principles (LDPs) of the invariant measures  of Langevin equations and their numerical methods, as the noise intensity $\epsilon\to 0$ and the dissipation intensity $\nu\to\infty$ respectively. First, by proving the weak LDP and the exponential tightness, we conclude that the invariant measure $\{\mu_{\nu,\epsilon}\}$ of the exact solution satisfies the LDPs as $\epsilon\to0$ and $\nu\to\infty$ respectively. 
	Then, we study whether there exist numerical methods asymptotically preserving these two LDPs of  $\{\mu_{\nu,\epsilon}\}$ in the sense that the rate functions of invariant measures of numerical methods converge pointwise to the rate function of $\{\mu_{\nu,\epsilon}\}$ as the step-size tends to zero. The answer is positive for the linear Langevin equation.
	For the small noise case, we show that a large class of numerical methods can asymptotically preserve the LDP of $\{\mu_{\nu,\epsilon}\}_{\epsilon>0}$ as $\epsilon\to0$. For the strong dissipation case, we study the stochastic $\theta$-method ($\theta\in[1/2,1]$) and show that only the  midpoint scheme ($\theta=1/2$) can asymptotically preserve the LDP of $\{\mu_{\nu,\epsilon}\}_{\nu>0}$ as $\nu\to\infty$. These results indicate that in the linear case, the LDP  as $\epsilon\to0$ and the LDP as $\nu\to\infty$ for the invariant measures of numerical methods have intrinsic differences: 
	the common numerical methods can asymptotically preserve the LDP of $\{\mu_{\nu,\epsilon}\}_{\epsilon>0}$ as $\epsilon\to0$ while the asymptotical preservation of numerical methods for the LDP of $\{\mu_{\nu,\epsilon}\}_{\nu>0}$ as $\nu\to\infty$ depends on the choice of numerical methods. To the best of our knowledge, this is the first result of investigating the relationship between the LDPs of invariant measures of stochastic differential equations and those of their numerical methods.
\end{abstract}

\maketitle

\section{Introduction}
\label{Sec1}Langevin equations have been widely applied to various systems driven by stochastic forcing, such as  chemical interactions, molecular simulations and quantum systems (see e.g. \cite{Langevin2,Langevin1}). For example, they are used to  describe the noise-induced transport
in stochastic ratchets and  dissipative particle dynamics
(see e.g. \cite{Langevin3} and references therein).
In this paper, we consider the  following $2$-dimensional  Langevin equation with deterministic initial value $(P(0),Q(0))=(p,q)$:
\begin{align}\label{Lan}
\ud P(t)&=-\nabla V(Q(t))\ud t-\nu P(t)\ud t+\sqrt{\epsilon}\ud W(t),\nonumber\\
\ud Q(t)&=P(t)\ud t,
\end{align}
where $W$ is a one-dimensional standard Wiener process defined on a complete filtered probabilities $\left(\Omega,\mscr F,\{\mscr F_t\}_{t\geq0},\mbf P\right)$ with $\{\mscr F_t\}_{t\geq0}$ satisfying the usual conditions.
$V:\mbb R\to\mbb R$ is a smooth potential.  The constants $\nu>0$ and $\epsilon>0$ denote the dissipation intensity and noise intensity  respectively. The equation \eqref{LDP} can  describe the motion of a particle according to Newton's second law and is  subject to  stochastic forcing and friction (see e.g. \cite{Fokker}).  It is well-known that \eqref{Lan} is ergodic and  the unique invariant measure $\mu_{\nu,\epsilon}$ is the Boltzmann-Gibbs probability measure given by
\begin{align*}
\ud \mu_{\nu,\epsilon}=\frac{1}{Z_{\nu,\epsilon}}\exp\left\{-\frac{2\nu}{\epsilon}\left(\frac{1}{2}p^2+V(q)\right)\right\}\ud p\ud q,
\end{align*}
where $Z_{\nu,\epsilon}=\int_{\mbb R^2}e^{-\frac{2\nu}{\epsilon}\left(\frac{1}{2}p^2+V(q)\right)}\,\ud p\ud q$ is the normalizing constant (see e.g., \cite[Proposition 6.1]{Fokker}). In this work, we focus on  investigating the LDP of $\{\mu_{\nu,\epsilon}\}$ as $\epsilon\to0$ and $\nu\to\infty
$ respectively, and what kind of numerical methods can asymptotically preserve the LDP of $\{\mu_{\nu,\epsilon}\}$.

The invariant measure $\mu_{\nu,\epsilon}$ describes the 
long-time behavior of \eqref{Lan}. Further, one may be interested in the asymptotical behavior of $\mu_{\nu,\epsilon}$ as the parameter $\nu$ or $\epsilon$ tends to $0$ or $\infty$. For example, when the values that $\mu_{\nu,\epsilon}$ assigns to measurable subsets are of exponential-type estimates as $\epsilon\to0$, the LDP of $\{\mu_{\nu,\epsilon}\}$ is involved. Roughly, the LDP characterizes the asymptotical behavior of a family of probability measures on an exponential scale, which plays an  important role in many fields such as the statistical physics and information theory (see e.g. \cite{LDP08}). The LDPs of invariant measures of various SDEs as the noise intensity $\epsilon\to0$ are studied in recent years
(see e.g. \cite{ZC15,Rockner04,GPP13,MaXi18,Sower92} and references therein).  In most of their work, the general strategy to derive the LDP of the invariant measure is based on the sample path large deviations of the original equations and the characterization of the action function, also called quasi-potential. Considering that the invariant measure $\mu_{\nu,\epsilon}$ is formulated explicitly in our case, we give a direct derivation for the LDP of  $\{\mu_{\nu,\epsilon}\}$ not based on the sample path large deviations of \eqref{Lan}. We not only study the LDP of $\{\mu_{\nu,\epsilon}\}_{\epsilon>0}$ in the small noise limit but also the LDP of $\{\mu_{\nu,\epsilon}\}_{\nu>0}$ in the strong dissipation limit. Our idea is to prove that $\{\mu_{\nu,\epsilon}\}_{\nu>0}$ or $\{\mu_{\nu,\epsilon}\}_{\epsilon>0}$ satisfies a weak LDP and is exponentially tight.
In the derivation of the weak LDP  of $\{\mu_{\nu,\epsilon}\}_{\nu>0}$, the main difficulty lies in how to prove that the limit $\lim_{\nu\to\infty}\frac{1}{\nu}\ln Z_{\nu,\epsilon}$ exists. We solve this problem
by means of the exponential-type estimate of
$\int_{|q|\geq L}e^{-\nu V(q)}\,\ud q$. By further verifying  the exponential tightness of $\{\mu_{\nu,\epsilon}\}_{\nu>0}$, we obtain the full LDP of $\{\mu_{\nu,\epsilon}\}_{\nu>0}$. Finally, a scaling argument is used to derive the LDP of $\{\mu_{\nu,\epsilon}\}_{\epsilon>0}$ from $\{\mu_{\nu,\epsilon}\}_{\nu>0}$.

Concerning the numerical approximations of the invariant measure of \eqref{Lan}, most authors are devoted to  constructing  numerical methods which inherit the invariant measure of the Langevin equation \eqref{Lan}, and analyzing the errors between the invariant measures of numerical methods and the original one (see e.g. \cite{HSW,Mattingly,Talay} and references therein). 
As far as we know, up to now, there has not been any work that studies the LDPs of invariant measures of numerical methods for stochastic differential equations (SDEs) in the existing literatures, this is one of our motivations. In this work, we aim to derive the LDPs of invariant measures of numerical methods for \eqref{Lan} and study the relationship between the LDPs of $\{\mu_{\nu,\epsilon}\}$ and its numerical counterparts $\{\mu^h_{\nu,\epsilon}\}$ with $h$ being the time step-size. 
The considered LDPs of $\{\mu_{\nu,\epsilon}\}$ and $\{\mu^h_{\nu,\epsilon}\}$ include two kinds of cases: one is as $\epsilon\to0$ which is called the LDP in the small noise limit; the other is as $\nu\to\infty$ which is called the LDP in the strong dissipation limit.

For an ergodic numerical method for \eqref{Lan} that possesses a unique invariant measure $\mu_{\nu,\epsilon}^h$, two natural questions are: whether $\mu_{\nu,\epsilon}^h$ also satisfies the LDP with the rate function $I^h$ as $\epsilon\to0$ or $\nu\to\infty$; if it does,  whether $I^h$ can approximate well the rate function of $\{\mu_{\nu,\epsilon}\}_{\epsilon>0}$ or $\{\mu_{\nu,\epsilon}\}_{\nu>0}$. Following the ideas in \cite{LDPosc,LDPschro}, we introduce the concept of asymptotical preservation of numerical methods for the LDP of $\{\mu_{\nu,\epsilon}\}_{\epsilon>0}$ or $\{\mu_{\nu,\epsilon}\}_{\nu>0}$ in the sense that the rate function $I^h$ converges pointwise to the rate function of $\{\mu_{\nu,\epsilon}\}_{\epsilon>0}$ or $\{\mu_{\nu,\epsilon}\}_{\nu>0}$ as $h\to0$. Roughly speaking, that a numerical method $\{P_n,Q_n\}_{n\geq0}$ asymptotically preserves the LDP of $\{\mu_{\nu,\epsilon}\}$ means that the exponential decay of $\mu_{\nu,\epsilon}(A)$ can be well approximated by $\mu_{\nu,\epsilon}^h(A)$, provided that $h$ is small enough, for a given measurable set $A\subseteq\mbb R^2$.
For a general smooth  potential $V$, constructing ergodic numerical methods for \eqref{Lan} is still under investigation. In addition, deriving the LDPs of invariant measures of ergodic numerical methods is challenging due to the  following two reasons: one is
the lack of explicit expressions of the invariant measures; the second is that since  the sample path of $\{P_n,Q_n\}_{n\geq0}$ is not continuous, whether it is still feasible to derive the LDP of $\{\mu_{\nu,\epsilon}^h\}$ based on the sample path LDP of $\{P_n,Q_n\}_{n\geq0}$.
But for some cases, the question can be simplified.  For example,  consider the case $\nabla V(0)=0$. 
Using the local linearization with $\nabla V(q)\approx \nabla^2 V(0)q$, one gets 
\begin{align*}
	\ud \hat P(t)&=-\nabla^2 V(0)\hat  Q(t)\ud t-\nu \hat  P(t)\ud t+\sqrt{\epsilon}\ud W(t),\nonumber\\
	\ud \hat  Q(t)&=\hat P(t)\ud t,
\end{align*}
which is viewed as an  approximation of \eqref{Lan} and reveals some local properties of \eqref{Lan}. In addition, some Langevin equations have a global linearization (see e.g., \cite{Linear}).
Thus, it is reasonable to illustrate the asymptotical behavior of the invariant measure of \eqref{Lan} in the linear case.
In this work, for the LDPs of invariant measures of numerical methods, we restrict our discussions to the linear case with $V(q)=\frac{1}{2}q^2$.

For the small noise case, we consider a  large class of  numerical methods which has at least first order convergence in mean-square sense (see \eqref{Mthd}).  By studying the weak limit of the distribution of this class of numerical methods, we prove that these numerical methods admit an invariant measure $\mu_{\nu,\epsilon}^h$ with $\nu$ being fixed. Further, we prove that for sufficiently small $h$, $\{\mu_{\nu,\epsilon}^h\}_{\epsilon>0}$ satisfies an LDP with the rate function $I^h$ by means of G\"artner--Ellis theorem. Finally  the convergence of $I^h$ gives that this class of numerical methods can asymptotically preserve the LDP of $\{\mu_{\nu,\epsilon}\}_{\epsilon>0}$ in the  small noise limit. 
The LDP of the invariant measures of $\{\mu_{\nu,\epsilon}^h\}_{\epsilon>0}$ in the strong dissipation limit is quite different from the small noise case. On the one hand, we note that $A$ and $b$ in the method \eqref{Mthd} will depend on the parameter $\nu$ for common numerical methods. Hence, the logarithmic moment generating functions of $\{\mu_{\nu,\epsilon}^h\}_{\nu>0}$ can not be explicitly given for \eqref{Mthd} in this case. On the other hand, whether numerical methods can asymptotically preserve the LDP of $\{\mu_{\nu,\epsilon}^h\}_{\nu>0}$ depends on the choice of the numerical method. 
More precisely, we study  the LDP of $\{\mu_{\nu,\epsilon}^h\}_{\nu>0}$ of stochastic $\theta$-method, $\theta\in[1/2,1]$,  and find that only the midpoint scheme can asymptotically preserve the LDP of $\{\mu_{\nu,\epsilon}\}_{\nu>0}$. 
These results indicate that in the linear case, the common numerical methods can asymptotically preserve the LDP of $\{\mu_{\nu,\epsilon}\}_{\epsilon>0}$ as $\epsilon\to0$ while only some specific ones possess stability in asymptotically preserving the LDP of  $\{\mu_{\nu,\epsilon}\}_{\nu>0}$ as $\nu\to\infty$.
As we know, we are the first to study the asymptotical preservation of numerical methods for the LDPs of invariant measures of SDEs.

The rest of this paper is organized as follows. Section \ref{Sec2} gives some preliminaries about the basic concepts and useful tools concerning the LDP and invariant measure. Section \ref{Sec3} is devoted to deriving the LDPs of $\{\mu_{\nu,\epsilon}\}$ in the small noise limit and strong dissipation limit respectively, under the case that $\nabla V$ may not be  globally Lipschitz. In Section \ref{Sec4}, we prove that a large class of numerical methods can asymptotically preserve the LDP of $\{\mu_{\nu,\epsilon}\}_{\epsilon>0}$ as $\epsilon\to0$ for the linear case. We study the LDP of the invariant measure of the stochastic $\theta$-method with $\theta\in[1/2,1]$ in the linear case in Section \ref{Sec5}, and show that the midpoint scheme can asymptotically preserve the LDP of $\{\mu_{\nu,\epsilon}\}_{\nu>0}$ as $\nu\to\infty$. Finally, the conclusions and  future work are presented in Section \ref{Sec6}.

\section{Preliminaries} \label{Sec2}
In this section, we introduce some basic concepts and useful tools in the theory of large deviations, which can  be found in \cite{ChenX,Dembo}. In addition, we also introduce some relative knowledge upon the invariant measure (readers can refer to \cite{Brehier,HongW} and their references). 
Throughout this section, let $\mathcal{X}$ be a separable Banach space and  $\mathcal{X}^*$ its dual space. In addition, let $\mscr B(\mcal X)$ be the Borel-$\sigma$ algebra of $\mcal X$. We first give the definitions of rate functions and LDP.

\begin{Def}\label{ratefun}
	$I:\mcal X\rightarrow[0,\infty]$ is called a rate function, if it is lower semi-continuous, i.e., for each $a\in[0,\infty)$, the level set $I^{-1}([0,a])$ is a closed subset of $\mcal X$. If all level sets $I^{-1}([0,a])$, $a\in[0,\infty)$, are compact, then $I$ is called a good rate function.
\end{Def}

\begin{Def}\label{LDP}
	Let $I$ be a rate function and $\{\mu_\epsilon\}_{\epsilon>0}$ be a family of probability measures on $\left(\mcal X,\mscr B(\mcal X)\right)$. We say that $\{\mu_\epsilon\}_{\epsilon>0}$ satisfies a (full) LDP with the rate function $I$  if
	\begin{flalign}
	(\rm{LDP 1})\qquad \qquad&\liminf_{\epsilon\to 0}\epsilon\ln(\mu_\epsilon(U))\geq-\inf I(U)\qquad\text{for every open subset}~ U\subseteq \mcal X,\nonumber&\\
	(\rm{LDP 2})\qquad\qquad &\limsup_{\epsilon\to 0}\epsilon\ln(\mu_\epsilon(C))\leq-\inf I(C)\qquad\text{for every closed subset}~ C\subseteq \mcal X.&\nonumber
	\end{flalign}
In addition, if (LDP2)  holds  for every compact subset $C\subseteq \mcal X$, we call that $\{\mu_{\epsilon}\}_{\epsilon>0}$ satisfies a weak LDP with the rate function $I$.
\end{Def}

Let $\{X_\epsilon\}_{\epsilon>0}$
be a family of random variables from $\left(\Omega,\mathscr{F},\mbf P\right)$ to $(\mcal X,\mscr{B}(\mcal X))$. Similarly, $\{X_\epsilon\}_{\epsilon>0}$  is said to satisfy an LDP with the rate function $I$, if its distribution law $(\mbf P\circ X_\epsilon^{-1})_{\epsilon>0}$  satisfies  (LDP$1$) and (LDP$2$) in Definition \ref{LDP}  (see e.g., \cite{Dembo}). As is shown in Definition \ref{LDP}, the LDP characterizes the asymptotical behavior of a family of probabilities on an exponential scale.
In order to strengthen a weak LDP to a full LDP, one usually needs to verify the so called exponential tightness. 

\begin{Def}\label{exptight}
	A family of probability measures $\{\mu_{\epsilon}\}_{\epsilon>0}$ on $\mcal X$ is exponentially tight if for every $a<\infty$, there exists a compact set $K_{a}\subseteq \mcal X$ such that 
	\begin{align*}
		\limsup_{\epsilon\to 0}\epsilon\ln\mu_{\epsilon}(K_{a}^c)<-a.
	\end{align*}
\end{Def}

\begin{pro}\cite[Lemma 1.2.18]{Dembo}\label{weakfull}
If $\{\mu_{\epsilon}\}_{\epsilon>0}$ on $\mcal X$ is exponentially tight and satisfies a weak LDP with a rate function $I$, then  $\{\mu_{\epsilon}\}_{\epsilon>0}$ satisfies the (full) LDP on $\mcal X$ with the rate function $I$, and $I$ is a good rate function.
\end{pro}

\begin{pro} \cite[Lemma 1.2.15]{Dembo} \label{limsup}
	Let $N$ be a fixed integer. Then,  for every $a^i_{\epsilon}\geq0$,
	\begin{align*}
	\limsup_{\epsilon\to 0}\epsilon\ln\left(\sum_{i=1}^{N}a^i_{\epsilon}\right)=\underset{i=1,\ldots,N}{\max}\limsup_{\epsilon\to 0}\epsilon\ln a^i_{\epsilon}.
	\end{align*}
\end{pro}
Proposition \ref{limsup} is a useful tool in estimating (LDP1) and (LDP2). Next, we give the celebrated G\"artner--Ellis theorem, which plays an important role in establishing the LDPs for a family of  dependent random variables.

\begin{theo}\cite[Corollary 4.6.14]{Dembo}\label{GE}
	Let $\{\mu_\epsilon\}_{\epsilon>0}$ be an exponentially tight family of Borel probability measures on  $\mcal X$. Suppose the logarithmic moment generating function $\Lambda(\cdot)=\lim_{\epsilon\to 0}\epsilon\Lambda_{\mu_{\epsilon}}(\cdot/\epsilon)$ is finite valued and Gateaux differentiable, where $\Lambda_{\mu_{\epsilon}}(\lambda):=\ln\int_{\mcal X}e^{\lambda(x)}\mu_{\epsilon}(dx)$, $\lambda\in \mcal X^*$. Then  $\{\mu_\epsilon\}_{\epsilon>0}$ satisfies the LDP in $\mcal X$ with the convex, good rate function $\Lambda^*(x)=\underset{\lambda\in\mcal X^*}{\sup}\{\lambda(x)-\Lambda(\lambda)\}$.
\end{theo}


Let $d$ be a given positive integer and $X$ satisfy a $d$-dimensional stochastic differential equation. Denote by $X(t,x)$  the value of $X$ at time $t$ starting from $x\in\mbb R^d$. 
Let $\mbf B_b(\mbb R^d)$ denote the set of all bounded and measurable functions from $\mbb R^d$ to $\mbb R$, and $\mbf C_b(\mbb R^d)$ denote the set of all bounded and continuous functions from $\mbb R^d$ to $\mbb R$.
Define the transition semigroup $\{\mscr P_t\}_{t\geq 0}$ on $\mbf B_b(\mbb R^d)$ as 
\begin{align*}
	\mscr P_t\varphi(x)=\mbf E\varphi(X(t,x)),\qquad\forall\quad \varphi\in\mbf B_b(\mbb R^d).
\end{align*}
Then $\{\mscr P_t\}_{t\geq 0}$ is a Markov semigroup. In addition,  $\{\mscr P_t\}_{t\geq 0}$ is called \emph{Feller} if for any $t> 0$ and $\varphi\in\mbf C_b(\mbb R^d)$, one has $\mscr P_t\varphi\in\mbf C_b(\mbb R^d)$.

\begin{Def}\label{inv}
A probability measure on $\left(\mbb R^d,\mscr B(\mbb R^d)\right)$ is said to be invariant for 
$\{\mscr P_t\}_{t\geq 0}$ or $X$ if
\begin{align*}
	\int_{\mbb R^d}\mscr P_t\varphi(x)\,\mu(\ud x)=\int_{\mbb R^d}\varphi(x)\,\mu(\ud x),\qquad\forall \quad\varphi\in\mbf C_b(\mbb R^d)~\text{and}~t\geq0.
\end{align*}
\end{Def}

The following proposition gives a sufficient condition to ensure the existence and uniqueness of the invariant measure for $\{\mscr P_t\}_{t\geq 0}$.
\begin{pro}\label{pro2.8}
Assume that $\{\mscr P_t\}_{t\geq0}$ is Feller. If for any initial value $x\in\mbb R^d$, the law of $X(t,x)$ weakly converges to some probability measure $\mu$ on $\left(\mbb R^d,\mscr B(\mbb R^d)\right)$ as $t\to\infty$, then $\mu$ is the unique invariant measure of $\{\mscr P_t\}_{t\geq0}$ or $X$.
\end{pro}

\begin{proof}
According to the assumption on $\mu$, it holds that
\begin{align}\label{sec2k1}
 \lim_{t\to\infty}\mscr P_t\varphi(x)=\lim_{t\to\infty}\mbf E\varphi(X(t,x))=\int_{\mbb R^d}\varphi \,\mu(dx),\qquad\forall \quad\varphi\in\mbf C_b(\mbb R^d)~\text{and}~x\in\mbb R^d.
\end{align}
For any $s\geq0$, since $\{\mscr P_t\}_{t\geq0}$ is \emph{Feller}, $\mscr P_s\varphi\in\mbf C_b(\mbb R^d)$ for any $\varphi\in\mbf C_b(\mbb R^d)$. By the semigroup property of $\mscr P_t$, we have
\begin{align*}
\mscr P_t\mscr P_s\varphi(x)=\mscr P_{s+t}\varphi(x),\qquad\forall \quad\varphi\in\mbf C_b(\mbb R^d)~\text{and}~x\in\mbb R^d.
\end{align*}
In view of \eqref{sec2k1}, putting $t\rightarrow \infty$ on both sides of the above identity produces
\begin{align*}
 \int_{\mbb R^d}\mscr P_s\varphi(x)\,\mu (dx)=\int_{\mbb R^d}\varphi(x)\,\mu (dx),\qquad\forall\quad \varphi\in\mbf C_b(\mbb R^d)~\text{and}~s\geq 0,
\end{align*}
i.e., $\mu$ is an invariant measure for $\{\mscr P_t\}_{t\geq0}$.

Assume that $\widetilde{\mu}$ is another invariant measure for $\{\mscr P_t\}_{t\geq0}$, then
\begin{align}\label{sec2k2}
\int_{\mbb R^d}\mscr P_t\varphi(x)\,\widetilde{\mu}(\ud x)=\int_{\mbb R^d}\varphi(x)\,\widetilde{\mu}(\ud x),\qquad\forall\quad \varphi\in\mbf C_b(\mbb R^d)~\text{and}~t\geq0.
\end{align}
Letting $t\to\infty$ in \eqref{sec2k2}, along with the dominated convergence theorem and \eqref{sec2k1}, yields
\begin{align*}
 \int_{\mbb R^d}\varphi(x)\,\mu(\ud x)=\int_{\mbb R^d}\varphi(x)\,\widetilde{\mu}(\ud x),\qquad\forall\quad \varphi\in\mbf C_b(\mbb R^d),
\end{align*}
where we have used the fact $\|\mscr P_t\|_{\mcal L(\mbf B_b(\mbb R^d))}\leq 1$ for any $t\geq 0$. 
The above formula implies $\mu=\widetilde{\mu}$,
which completes the proof.
\end{proof}
If $\{X_n\}$ is a discrete Markov's chain, for example, the numerical approximation  of $X$, corresponding discrete versions of  Definition \ref{inv} and Proposition \ref{pro2.8} are also valid.

\section{LDPs for the invariant measure of stochastic Langevin equation}\label{Sec3}

Throughout this paper, we assume that $V\in\mbf C^{\infty}\left(\mbb R,\mbb R\right)$ is a \emph{confining potential}, i.e., $\lim_{|q|\to+\infty}V(q)=+\infty$ and $e^{-\beta V(q)}\in  L^1(\mbb R)$ for any $\beta\in \mbb R^+$. In this case, there is some constant $C$ such that $V(q)+C\geq 1$ for any $q\in\mbb R$ and $V+C$ is also a confining potential.  As a consequence, $V+C$ satisfies the assumptions of \cite{GGMZ}, which implies that \eqref{Lan} admits a unique strong solution since $\nabla V=\nabla (V+C)$. 
Recall that \eqref{Lan} possesses a unique invariant measure $\mu_{\nu,\epsilon}$ given by
\begin{align*}
	\ud \mu_{\nu,\epsilon}=\frac{1}{Z_{\nu,\epsilon}}\exp\left\{-\frac{2\nu}{\epsilon}\left(\frac{1}{2}p^2+V(q)\right)\right\}\ud p\ud q,
\end{align*}
where $$Z_{\nu,\epsilon}=\int_{\mbb R^2}e^{-\frac{2\nu}{\epsilon}\left(\frac{1}{2}p^2+V(q)\right)}\,\ud p\ud q=\sqrt{\frac{\pi\epsilon}{\nu}}\int_{\mbb R}e^{-\frac{2\nu}{\epsilon} V(q)}\ud q.$$
In what follows, we study two kinds of LDPs of $\left\{\mu_{\nu,\epsilon}\right\}$: one is the LDP of $\left\{\mu_{\nu,\epsilon}\right\}_{\nu>0}$ as $\nu$ tends to infinity for any given $\epsilon>0$; the second one is the LDP of  $\left\{\mu_{\nu,\epsilon}\right\}_{\epsilon>0}$ as $\epsilon$ tends to zero for any given $\nu>0$. For this end, we need further assumption on $V$.

\textbf{Assumption 1.} Assume $V\in\mbf C^{\infty}(\mbb R,\mbb R)$ and there exist $\eta,\alpha>0$, $L_0\geq1$ such that $V(q)\geq\eta|q|^{\alpha}$ for all $|q|\geq L_0$. 

It is verified that under \textbf{Assumption 1}, $V$ is a confining potential. Define $I(p,q) := p^2 + 2V (q) $ for any $p,q \in\mbb R$.
Based on this assumption, we present the main results of this section.
\begin{lem}\label{lem3.1}
	Under \textbf{Assumption 1}, it holds that
	$$\lim_{\nu\to\infty}\frac{1}{\nu}\ln Z_{\nu,1}= \lim_{\nu\to\infty}\frac{1}{\nu}\ln\int_{\mbb R}e^{-2\nu V(q)}\ud q=Z_0,$$
	where $Z_0:=-2\inf_{q\in\mbb R}V(q)$. 
	\end{lem}

\begin{theo}\label{tho3.2}
Under \textbf{Assumption 1}, for any $\epsilon>0$, $\left\{\mu_{\nu,\epsilon}\right\}_{\nu>0}$ satisfies an LDP on $\mbb R^{2}$ with the good rate function $(I+Z_0)/\epsilon$, i.e.,
\begin{align}
	\liminf_{\nu\to\infty}\frac{1}{\nu}\ln \mu_{\nu,\epsilon}(U)\geq& -\inf_{(p,q)\in U}\frac{1}{\epsilon} \left(I(p,q)+Z_0\right)\qquad \text{for every open subset}~  U\subseteq\mbb R^2, \label{lowLDP}\\
	\limsup_{\nu\to\infty}\frac{1}{\nu}\ln \mu_{\nu,\epsilon}(C)\leq& -\inf_{(p,q)\in C} \frac{1}{\epsilon} \left(I(p,q)+Z_0\right)\qquad \text{for every closed subset}~  C\subseteq\mbb R^2. \label{upLDP}
\end{align}	
\end{theo}

\begin{theo}\label{tho3.3}
Under \textbf{Assumption 1}, for any $\nu>0$,	$\left\{\mu_{\nu,\epsilon}\right\}_{\epsilon>0}$ satisfies an LDP on $\mbb R^{2}$ with the good rate function $\nu(I+Z_0)$, i.e.,
	\begin{align*}
	\liminf_{\epsilon\to0}\epsilon\ln \mu_{\nu,\epsilon}(U)\geq& -\inf_{(p,q)\in U} \nu(I(p,q)+Z_0)\qquad \text{for every open subset}~  U\subseteq\mbb R^2,\nonumber\\
	\limsup_{\epsilon\to0}\epsilon\ln \mu_{\nu,\epsilon}(C)\leq& -\inf_{(p,q)\in C} \nu(I(p,q)+Z_0)\qquad \text{for every closed subset}~  C\subseteq\mbb R^2.
	\end{align*}	
\end{theo}

\subsection{Proof of Lemma \ref{lem3.1}}
In this part, we aim to prove Lemma \ref{lem3.1} by means of some useful lemmas. First of all, we note the following fact.
\begin{lem}\label{lem3.4}
	If $f:\mbb R\to\mbb R$ is upper semi-continuous, then for any $a,b\in\mbb R$ with $a<b$,
	\begin{align*}
		\lim_{\nu\to\infty}\frac{1}{\nu}\ln\int_{a}^{b}e^{-\nu f(q)}\ud q=-\inf_{q\in[a,b]}f(q).
	\end{align*}
\end{lem}
\begin{proof}
	Let $q_0\in[a,b]$ be fixed. Since $f$ is upper semi-continuous, for any $\varepsilon>0$, there is some $\delta>0$ such that
	\begin{align*}
		f(q)<f(q_0)+\varepsilon, \qquad\forall\quad q\in B(q_0,\delta). 
	\end{align*}
Thus, we have
\begin{align*}
	\liminf_{\nu\to\infty}\frac{1}{\nu}\ln\int_{a}^{b}e^{-\nu f(q)}\ud q\geq& \liminf_{\nu\to\infty}\frac{1}{\nu}\ln\int_{B(q_0,\delta)\cap [a,b]}e^{-\nu f(q)}\ud q \nonumber\\
	\geq & \liminf_{\nu\to\infty}\frac{1}{\nu}\ln\left(e^{-\nu (f(q_0)+\varepsilon)}\big|B(q_0,\delta)\cap [a,b]\big|\right)\nonumber \\
	=&-f(q_0)+\varepsilon,
\end{align*}
where $|B(q_0,\delta)\cap [a,b]|$ denotes the Lebesgue measure of $B(q_0,\delta)\cap [a,b]$.
Taking $\varepsilon\to 0$ in the above formula and using the arbitrariness of $q_0$ yield
\begin{align}\label{sec3k1}
\liminf_{\nu\to\infty}\frac{1}{\nu}\ln\int_{a}^{b}e^{-\nu f(q)}\ud q\geq \sup_{q\in[a,b]}-f(q)=-\inf_{q\in[a,b]}f(q).
\end{align}
Further, we have 
\begin{align}\label{sec3k2}
	\limsup_{\nu\to\infty}\frac{1}{\nu}\ln\int_{a}^{b}e^{-\nu f(q)}\ud q\leq \limsup_{\nu\to\infty}\frac{1}{\nu}\ln\left(e^{-\nu\inf_{q\in[a,b]}f(q)}(b-a)\right)=-\inf_{q\in[a,b]}f(q).
\end{align}
Combining \eqref{sec3k1} and \eqref{sec3k2}, we complete the proof.
\end{proof}

\begin{lem}\label{lem3.5}
	Let $f:\mbb R\to\mbb R$ be a measurable function. Assume that there exist $\eta,\alpha>0$ and $L_0\geq1$ such that
	\begin{align}\label{sec3k3}
		f(q)\geq\eta |q|^{\alpha},\qquad\forall\quad |q|\geq L_0.
	\end{align}
	Then for each $L\geq L_0$,
\begin{align}\label{sec3k4}
	\limsup_{\nu\to\infty}\frac{1}{\nu}\ln\int_{|q|\geq L}e^{-\nu f(q)}\ud q\leq -\frac{\eta L^{\alpha
	}}{2}.
\end{align}
\end{lem}
\begin{proof}
	\textit{Step 1: We prove the conclusion for $\alpha=2$}.\\
	In this case, \eqref{sec3k3} becomes
	\begin{align*}
	f(q)\geq\eta |q|^{2},\qquad\forall\quad |q|\geq L_0.
	\end{align*}
Denote $S_{L,\nu}:=\int_{|q|\geq L}e^{-\nu f(q)}\ud q$ for each $L\geq L_0$.
Then $S_{L,\nu}\leq \int_{|q|\geq L}e^{-\nu\eta q^2}\ud q$, which leads to
\begin{align*}
	S_{L,\nu}^2\leq & \int_{|q_1|\geq L\,,|q_2|\geq L}  e^{-\nu\eta(q_1^2+q_2^2)}\ud q_1\ud q_2 
	\leq \int_{q_1^2+q_2^2\geq L^2}e^{-\nu\eta(q_1^2+q_2^2)}\ud q_1\ud q_2\\
	= &\int_{0}^{2\pi}\int_{L}^{+\infty}e^{-\nu\eta r^2}r\ud r\ud \theta 
	=\pi \int_{L}^{+\infty}e^{-\nu\eta r^2}\ud r^2 \\
	=&\frac{\pi}{\nu\eta}e^{-\nu\eta L^2}.
\end{align*}
As a consequence, we obtain
\begin{align}\label{sec3k5}
	\limsup_{\nu\to\infty}\frac{1}{\nu}\ln S_{L,\nu}=\limsup_{\nu\to\infty}\frac{1}{2\nu}\ln S_{L,\nu}^2\leq\limsup_{\nu\to\infty} \frac{1}{2\nu}\ln\left(\frac{\pi}{\nu\eta}e^{-\nu\eta L^2}\right)=-\frac{\eta L^2}{2}.
	\end{align}
	This proves \eqref{sec3k4} for $\alpha=2$.
	
\textit{Step 2: We prove the conclusion for $\alpha\neq 2$.}
Using \eqref{sec3k3} and the variable substitution $x=q^{\frac{\alpha}{2}}$, we have that for each $L\geq L_0$,
\begin{align}\label{sec3k6}
	S_{L,\nu}\leq \int_{|q|\geq L}e^{-\nu\eta |q|^{\alpha}}\ud q=2 \int_{q\geq L}e^{-\nu\eta q^{\alpha}}\ud q=\frac{4}{\alpha}\int_{x\geq L^{\frac{\alpha}{2}}}e^{-\nu\eta x^{2}}x^{\frac{2}{\alpha}-1}\ud x.
\end{align}
Since for each $\alpha>0$, $\lim_{x\rightarrow\infty}x^{\frac{2}{\alpha}-1}e^{-x^2}=0$, we have that $x^{\frac{2}{\alpha}-1}e^{-x^2}$ is bounded on $[1,+\infty)$. This implies that for each $\alpha>0$, there is some constant $C_{\alpha}>0$ such that 
\begin{align}\label{sec3claim}
	x^{\frac{2}{\alpha}-1}\leq C_{\alpha}e^{x^2},\qquad\forall\quad x\geq 1.
\end{align}
Substituting 
\eqref{sec3claim} into \eqref{sec3k6} gives
\begin{align*}
	S_{L,\nu}\leq \frac{4}{\alpha}C_{\alpha}\int_{x\geq L^{\frac{\alpha}{2}}}e^{-\nu\eta x^2}e^{x^2}\ud x=\frac{2}{\alpha}C_{\alpha}\int_{|x|\geq L^{\frac{\alpha}{2}}}e^{-(\nu\eta-1) x^2}\ud x.
\end{align*}
It follows from the above formula and \eqref{sec3k5} that
\begin{align*}
	\limsup_{\nu\to\infty}\frac{1}{\nu}\ln S_{L,\nu}\leq&\limsup_{\nu\to\infty}\frac{1}{\nu}\ln\int_{|x|\geq L^{\frac{\alpha}{2}}}e^{-(\nu\eta-1) x^2}\ud x\\
	=&\limsup_{\nu\to\infty}\frac{\nu-\frac{1}{\eta}}{\nu}\frac{1}{\nu-\frac{1}{\eta}}\ln \int_{|x|\geq L^{\frac{\alpha}{2}}}e^{-\eta(\nu-\frac{1}{\eta}) x^2}\ud x\\
	=&\lim_{t\to\infty}\frac{1}{t}\ln\int_{|x|\geq L^{\frac{\alpha}{2}}}e^{-t\eta x^2}\ud x\\
	\leq&-\frac{\eta L^{\alpha}}{2}.
\end{align*}
Combining the above discussions, we complete the proof.
\end{proof}
\begin{lem}\label{lem3.6}
	Let $f:\mbb R\to\mbb R$ be an upper semi-continuous function. If there exist $\eta,\alpha>0$ and $L_0\geq 1$ such that
	\begin{align*}
	f(q)\geq\eta |q|^{\alpha},\qquad\forall\quad |q|\geq L_0,
	\end{align*}
	then 
	\begin{align*}
	\lim_{\nu\to\infty}\frac{1}{\nu}\ln \int_{\mbb R}e^{-\nu f(q)}\ud q=-\inf_{q\in\mbb R} f(q).
	\end{align*}
\end{lem}
\begin{proof}
	First, it follows from Proposition \ref{limsup}, Lemmas \ref{lem3.4} and \ref{lem3.5} that for any $L\geq L_0$,
	\begin{align*}
	\limsup_{\nu\to\infty}\frac{1}{\nu}\ln \int_{\mbb R}e^{-\nu f(q)}\ud q=&\max\left\{\limsup_{\nu\to\infty}\frac{1}{\nu}\ln \int_{|q|\leq L}e^{-\nu f(q)}\ud q,\limsup_{\nu\to\infty}\frac{1}{\nu}\ln \int_{|q|\geq L}e^{-\nu f(q)}\ud q\right\}\\
	\leq&\max\left\{-\inf_{|q|\leq L}f(q),-\frac{\eta L^{\alpha}}{2}\right\}.
	\end{align*}
	Notice that $\frac{\eta L^{\alpha}}{2}\geq f(0)$ holds true for any $L\geq L_1:=\max\{L_0,(2|f(0)|/\eta)^{1/\alpha}\}$. Let $L\ge L_1$ be arbitrarily fixed. 
	Next, notice that $-\frac{\eta L^{\alpha}}{2}\leq -f(0)\leq -\inf_{|q|\leq L}f(q)$, which indicates that
	\begin{align}\label{sec3k7}
	\limsup_{\nu\to\infty}\frac{1}{\nu}\ln \int_{\mbb R}e^{-\nu f(q)}\ud q\leq -\inf_{|q|\leq L}f(q).
	\end{align}
	On the other hand, by Lemma \ref{lem3.4}, 
	\begin{align}\label{sec3k8}
	\liminf_{\nu\to\infty}\frac{1}{\nu}\ln \int_{\mbb R}e^{-\nu f(q)}\ud q\geq& \liminf_{\nu\to\infty}\frac{1}{\nu}\ln \int_{|q|\leq L}e^{-\nu f(q)}\ud q=-\inf_{|q|\leq L}f(q).
	\end{align}
	According to the assumptions on $f$, 
	$\inf_{|q|\geq L}f(q)\geq\inf_{|q|\geq L}\eta|q|^{\alpha}\geq \eta L^{\alpha}\ge f(0)\ge \inf_{|q|\leq L}f(q),$ 
	from which it follows that
	$$\inf_{q\in\mbb R}f(q)=\min\left\{\inf_{|q|\leq L}f(q),\inf_{|q|\geq L}f(q)\right\}=\inf_{|q|\leq L}f(q).$$ 
	This together with \eqref{sec3k7} and \eqref{sec3k8} finally yields
	\begin{align*}
	\lim_{\nu\to\infty}\frac{1}{\nu}\ln\int_{\mbb R}e^{-\nu f(q)}\ud q=-\inf_{|q|\leq L}f(q)=-\inf_{q\in\mbb R}f(q).
	\end{align*}
	This completes the proof.
\end{proof}

We are now in the position to prove Lemma \ref{lem3.1} based on the above  lemma.

\vspace{2mm}
\noindent{\large\textbf{Proof of Lemma \ref{lem3.1}.}}
Recall $Z_{\nu,1}=\sqrt{\frac{\pi}{\nu}}\int_{\mbb R}e^{-2\nu V(q)}\ud q$. Under \textbf{Assumption 1}, it follows from Lemma \ref{lem3.6} that
\begin{align*}
\lim_{\nu\to\infty}\frac{1}{\nu}\ln Z_{\nu,1}=\lim_{\nu\to\infty}\frac{1}{\nu}\ln\int_{\mbb R}e^{-2\nu V(q)}\ud q=2\lim_{\nu\to\infty}\frac{1}{\nu}\ln\int_{\mbb R}e^{-\nu V(q)}\ud q=-2\inf_{q\in\mbb R}V(q).
\end{align*} 
Hence, we complete the proof. \hfill $\square$

\subsection{Proofs of Theorems \ref{tho3.2} and \ref{tho3.3}}
In this part, we give the proofs of Theorems \ref{tho3.2} and \ref{tho3.3}. 


\vspace{2mm}
\noindent{\large\textbf{Proof of Theorem \ref{tho3.2}.}}

\textit{Step 1: We show that  $\{\mu_{\nu,1}\}_{\nu>0}$ satisfies a weak LDP with the rate function $I+Z_0$.}

Clearly, under \textbf{Assumption 1}, $I+Z_0$ is a rate function.
Let  $U\subseteq \mbb R^2$ be a given non-empty open set. Since $U$ is open and $I$ is continuous, for any fixed $x_0\in U$ and $\varepsilon>0$, there exists some $\delta>0$ such that
\begin{align*}
B(x_0,\delta)\subseteq U\quad\text{and}\quad I(x)<I(x_0)+\varepsilon, \quad\forall\quad x\in B(x_0,\delta).
\end{align*}
By the above formula and Lemma \ref{lem3.1},
\begin{align*}
\liminf_{\nu\to\infty}\frac{1}{\nu}\ln\mu_{\nu,1}(U)\geq& \liminf_{\nu\to\infty}\frac{1}{\nu}\ln\mu_{\nu,1}(B(x_0,\delta))\nonumber\\
= &\liminf_{\nu\to\infty}\frac{1}{\nu}\ln\left(\frac{1}{Z_{\nu,1}}\int_{B(x_0,\delta)}e^{-\nu I(p,q)}\ud p\ud q\right)\nonumber\\
\geq &-Z_0+\liminf_{\nu\to\infty}\frac{1}{\nu}\ln\left(e^{-\nu(I(x_0)+\varepsilon)}\big|B(x_0,\delta)\big|\right)\nonumber\\
=&-Z_0-I(x_0)-\varepsilon.
\end{align*}
Letting $\varepsilon\to 0$ in the above formula and using the arbitrariness of $x_0$, we have
\begin{align}\label{sec3k9}
\liminf_{\nu\to\infty}\frac{1}{\nu}\ln\mu_{\nu,1}(U)\geq\sup_{x\in U}\left\{-I(x)-Z_0\right\}=-\inf_{(p,q)\in U}\left\{I(p,q)+Z_0\right\}.
\end{align}
In addition, notice that \eqref{sec3k9} holds naturally for $U=\emptyset$. These imply that 
\begin{align}\label{sec3k10}
\liminf_{\nu\to\infty}\frac{1}{\nu}\ln\mu_{\nu,1}(U)\geq-\inf_{(p,q)\in U}\left\{I(p,q)+Z_0\right\}\qquad\text{for every open set}~U\subseteq \mbb R^2.
\end{align}

For any compact set $C\subseteq \mbb R^2$, we have that $C$ is bounded and hence $|C|<+\infty$. If $|C|=0$, then we have $\mu_{\nu,1}(C)=0$. In this case, $\limsup_{\nu\to\infty}\frac{1}{\nu}\ln\mu_{\nu,1}(C)=-\infty\leq-\inf_{(p,q)\in C}(I(p,q)+Z_0)$. If $0<|C|<+\infty$, then by Lemma \ref{lem3.1},
\begin{align*}
\limsup_{\nu\to\infty}\frac{1}{\nu}\ln\mu_{\nu,1}(C)=  &\liminf_{\nu\to\infty}\frac{1}{\nu}\ln\left(\frac{1}{Z_{\nu,1}}\int_{C}e^{-\nu I(p,q)}\ud p\ud q\right) \\
\leq& -Z_0+\liminf_{\nu\to\infty}\frac{1}{\nu}\ln\left(e^{-\nu\inf_{x\in C}I(x)}|C|\right)\\
=&-Z_0-\inf_{x\in C}I(x)=-\inf_{x\in C}(I(x)+Z_0).
\end{align*}
Thus, it holds that
\begin{align}\label{sec3k11}
\limsup_{\nu\to\infty}\frac{1}{\nu}\ln\mu_{\nu,1}(C)\leq-\inf_{(p,q)\in C}\left\{I(p,q)+Z_0\right\}\qquad\text{for every compact set}~C\subseteq \mbb R^2.
\end{align}
According to Definition \ref{LDP}, we obtain that $\{\mu_{\nu,1}\}_{\nu>0}$ satisfies a weak LDP with the rate function $I+Z_0$.

\textit{Step 2: We show that  $\{\mu_{\nu,1}\}_{\nu>0}$ is exponentially tight.}

Denote $K_L:=[-L,L]^2$, $L>0$, which is a compact set of $\mbb R^2$.
For $L\geq L_0$, using the fact $K_L^c=\left([-L,L]^c\times\mbb R\right)\cup\left(\mbb R\times[-L,L]^c\right)$, Lemmas \ref{limsup} and \ref{lem3.1}  yields
\begin{align*}
	&\limsup_{\nu\to\infty}\frac{1}{\nu}\ln\mu_{\nu,1}(K_L^c)\\
	\leq&-Z_0+\limsup_{\nu\to\infty}\frac{1}{\nu}\ln\left(\int_{|p|\geq L}e^{-\nu p^2}\ud p\int_{\mbb R}e^{-2\nu V(q)}\ud q+\int_{\mbb R}e^{-\nu p^2}\ud p\int_{|q|\geq\ L}e^{-2\nu V(q)}\ud q\right)\\
	=&-Z_0+\max\left\{\limsup_{\nu\to\infty}\frac{1}{\nu}\ln\int_{|p|\geq L}e^{-\nu p^2}\ud p\int_{\mbb R}e^{-2\nu V(q)}\ud q,\limsup_{\nu\to\infty}\frac{1}{\nu}\ln\sqrt{\frac{\pi}{\nu}}\int_{|q|\geq\ L}e^{-2\nu V(q)}\ud q\right\}\\
	=&-Z_0+\max\left\{Z_0+\limsup_{\nu\to\infty}\frac{1}{\nu}\ln\int_{|p|\geq L}e^{-\nu p^2}\ud p,\limsup_{\nu\to\infty}\frac{1}{\nu}\ln\int_{|q|\geq\ L}e^{-2\nu V(q)}\ud q\right\}.
\end{align*}
Due to Lemma \ref{lem3.5},
\begin{gather*}
	\limsup_{\nu\to\infty}\frac{1}{\nu}\ln\int_{|p|\geq\ L}e^{-\nu p^2}\ud p\leq -\frac{L^2}{2},\\
	\limsup_{\nu\to\infty}\frac{1}{\nu}\ln\int_{|q|\geq\ L}e^{-2\nu V(q)}\ud q\leq -\eta L^\alpha.
\end{gather*} 
Combining the above formulas, we obtain 
\begin{align*}
	\limsup_{\nu\to\infty}\frac{1}{\nu}\ln\mu_{\nu,1}(K_L^c)\leq-Z_0+\max\left\{Z_0-\frac{L^2}{2},-\eta L^\alpha \right\}=\max\left\{-\frac{L^2}{2},-\eta L^\alpha -Z_0\right\}.
\end{align*}
Hence, we obtain
\begin{align*}
	\lim_{L\to\infty}\limsup_{\nu\to\infty}\frac{1}{\nu}\ln\mu_{\nu,1}(K_L^c)=-\infty,
\end{align*}  
which implies the exponential tightness of $\{\mu_{\nu,1}\}_{\nu>0}$.
 From Proposition \ref{weakfull}, it follows that $\{\mu_{\nu,1}\}_{\nu>0}$ satisfies an LDP with the good rate function $I+Z_0$.

\textit{Step 3: We show that for every $\epsilon>0$, $\{\mu_{\nu,\epsilon}\}_{\nu>0}$ satisfies an LDP with the good rate function $(I+Z_0)/\epsilon$.}

Note that  $\mu_{\nu,\epsilon}=\mu_{\nu/\epsilon,1}$ for every $\epsilon>0$. Let $t=\nu/\epsilon$. Then by the LDP for $\{\mu_{t,1}\}_{t>0}$, we have that
for every open set $U\subseteq \mbb R^2$,
\begin{gather*}
	\liminf_{\nu\to\infty}\frac{1}{\nu}\ln\mu_{\nu,\epsilon}(U)=\liminf_{\nu\to\infty}\frac{1}{\nu}\ln\mu_{\nu/\epsilon,1}(U)=\frac{1}{\epsilon}\liminf_{t\to\infty}\frac{1}{t}\ln\mu_{t,1}(U)\ge-\frac{1}{\epsilon}\inf_{(p,q)\in U}
	(I(p,q)+Z_0),
\end{gather*}
and for every closed set $C\subseteq \mbb R^2$,
\begin{gather*}
	\limsup_{\nu\to\infty}\frac{1}{\nu}\ln\mu_{\nu,\epsilon}(C)=\frac{1}{\epsilon}\limsup_{t\to\infty}\frac{1}{t}\ln\mu_{t,1}(C)\le-\frac{1}{\epsilon}\inf_{(p,q)\in C}(I(p,q)+Z_0).
\end{gather*}
The proof is completed.
\hfill $\square$

\vspace{2mm}

\noindent{\large\textbf{Proof of Theorem \ref{tho3.3}}.}

For any measurable set $F\subseteq\mbb R^2$, using $\mu_{\nu,\epsilon}=\mu_{\nu/\epsilon,1}$ gives
\begin{align*}
	\liminf_{\epsilon\to0}\epsilon\ln\mu_{\nu,\epsilon}(F)=&\liminf_{\epsilon\to0}\epsilon\ln\mu_{\nu/\epsilon,1}(F)=\nu\liminf_{t\to\infty}\frac{1}{t}\ln\mu_{t,1}(F),\\
	\limsup_{\epsilon\to0}\epsilon\ln\mu_{\nu,\epsilon}(F)=&\limsup_{\epsilon\to0}\epsilon\ln\mu_{\nu/\epsilon,1}(F)=\nu\limsup_{t\to\infty}\frac{1}{t}\ln\mu_{t,1}(F).
\end{align*}
Based on Theorem \ref{tho3.2} and the above formulas, we finish the proof. \hfill$\square$

\section{LDPs of invariant measures of numerical methods with small noise}\label{Sec4}
Let $\{P_n,Q_n\}_{n\geq0}$ be a numerical method  for \eqref{Lan}, i.e., $(P_n,Q_n)$ is used to approximate   $(P(t_n),Q(t_n))$, where $t_n=nh$, $n=0,1,2,\ldots$, and $h>0$ is the step-size.  
If  the numerical method $\{P_n,Q_n\}_{n\geq0}$ possesses a unique invariant measure $\mu_{\nu,\epsilon}^h$, one may ask whether $\{\mu_{\nu,\epsilon}^h\}$ satisfies two kinds of the LDPs as in the continuous case when  $\epsilon\to0$ (resp. $\nu$ and $h$ are fixed) or $\nu\to\infty$ (resp. $\epsilon$ and $h$ are fixed). Further, if $\{\mu_{\nu,\epsilon}^h\}_{\epsilon>0}$ \big(resp. $\{\mu_{\nu,\epsilon}^h\}_{\nu>0}$\big) satisfies the LDP with the rate function $I^h$ (resp. $J^h$), whether $I^h$ (resp. $J^h$) can approximate well the rate function $\nu(I+Z_0)$ (resp. $(I+Z_0)/\epsilon$) of $\{\mu_{\nu,\epsilon}\}_{\epsilon>0}$ (resp. $\{\mu_{\nu,\epsilon}\}_{\nu>0}$) for sufficiently small step-size. For this end, we give the definition of numerically asymptotical preservation  for the LDP of $\{\mu_{\nu,\epsilon}\}_{\epsilon>0}$ or $\{\mu_{\nu,\epsilon}\}_{\nu>0}$.

\begin{Def}\label{asyLDP}
	Let $\{P_n,Q_n\}_{n\geq0}$ be a numerical method  for \eqref{Lan} with $(P_0,Q_0)=(P(0),Q(0))$ and $\nu>0$ (resp. $\epsilon>0$) be fixed. Assume that there is some $h_0>0$  such that for any $h\leq h_0$ and all sufficiently small $\epsilon>0$ (resp. all sufficiently large $\nu>0$), $\{P_n,Q_n\}_{n\geq0}$ possesses a unique invariant measure $\mu_{\nu,\epsilon}^h$. Further, assume that for any $h\leq h_0$, $\{\mu_{\nu,\epsilon}^h\}_{\epsilon>0}$ (resp. $\{\mu_{\nu,\epsilon}^h\}_{\nu>0}$) satisfies the LDP with the rate function $I^h$ (resp. $J^h$). The numerical method $\{P_n,Q_n\}_{n\geq0}$ is said to asymptotically preserve the LDP of $\{\mu_{\nu,\epsilon}\}_{\epsilon>0}$ (resp. $\{\mu_{\nu,\epsilon}^h\}_{\nu>0}$) if for any $(p,q)\in\mbb R^2$,
	\begin{align*}
	\lim_{h\to0}I^h(p,q)=\nu(I(p,q)+Z_0) \qquad\big(resp.\, \lim_{h\to0}J^h(p,q)=(I(p,q)+Z_0)/\epsilon\big).
	\end{align*}
\end{Def}


 In this section, we focus on the linear case with $V(q)=\frac{1}{2}q^2$. We show that general numerical methods can asymptotically preserve the LDP of $\{\mu_{\nu,\epsilon}\}_{\epsilon>0}$. And we will prove that the midpoint scheme can asymptotically preserve the LDP of $\{\mu_{\nu,\epsilon}\}_{\nu>0}$ in the next section.
For the case $V(q)=\frac{1}{2}q^2$, \textbf{Assumption 1} holds and $Z_0=0$, $I(p,q)=p^2+q^2$. Theorem \ref{tho3.3} shows that $\{\mu_{\nu,\epsilon}\}_{\epsilon>0}$ satisfies the LDP with the good rate function $\nu(p^2+q^2)$.
We consider the general numerical method  of the following form 
\begin{align}\label{Mthd}
\left(\begin{array}{c}
P_{n+1}\\\\
Q_{n+1}
\end{array}\right)=
\left(
\begin{array}{cc}
a_{11}(h)&a_{12}(h)\\\\
a_{21}(h)&a_{22}(h)
\end{array}\right)
\left(\begin{array}{cc}
P_n\\\\
Q_n
\end{array}\right)
+\sqrt{\epsilon} 
\left(\begin{array}{cc}
b_1(h)\\\\
b_2(h)
\end{array}\right)\Delta W_{n},\qquad n=0,1,2,\ldots
\end{align}
with $(P_0,Q_0)=(P(0),Q(0))=(p,q)$, where $\Delta W_{n}= W(t_{n+1})-W(t_n)$ with $t_n=nh$, $n=1,2,\ldots$, and  $a_{ij},\,b_i: (0,\infty)\to\mbb R$, $i,j=1,2$ are  the functions of step-size $h$ and determined by a concrete method. The functions $a_{ij},\,b_i$, $i,j=1,2$ depend on the parameter $\nu$ but are independent of $\epsilon>0$.

By defining functions
\begin{align*}
A(h):=\left(
\begin{array}{cc}
a_{11}(h)&a_{12}(h)\\
\\
a_{21}(h)&a_{22}(h)
\end{array}\right),\qquad b(h):=	\left(\begin{array}{cc}
b_1(h)\\\\
b_2(h)
\end{array}\right), \qquad\forall\quad h>0,
\end{align*}
we rewrite \eqref{Mthd} as
\begin{align}\label{sec4recur}
\left(\begin{array}{c}
P_{n+1}\\
Q_{n+1}
\end{array}\right)=
A(h)\left(\begin{array}{cc}
P_n\\
Q_n
\end{array}\right)
+\sqrt{\epsilon}
b(h)\Delta W_{n}, \qquad n=0,1,2,\ldots
\end{align}
Hereafter, we alway omit the argument $h$ if no confusion occurs. 

For given functions $f,g:(0,+\infty)\to \mbb R$,  $f(h)=\mcal O(h^p)$  stands for $|f(h)|\leq Ch^p$ for all sufficiently small $h>0$ and $f(h)=g(h)+\mcal O(h^p)$ stands for $f(h)-g(h)=\mcal O(h^p)$, where $C$ is a positive constant  independent of $h$.
 In addition, $f(h)\sim g(h)$ means $\lim_{h\to0}f(h)/g(h)=1$.
In order to show that \eqref{Mthd} asymptotically preserves the LDP of $\{\mu_{\nu,\epsilon}\}_{\epsilon>0}$, we impose the following assumption.

\vspace{2mm}
{\large\textbf{Assumption 2.}} The functions $a_{i,j}\,,b_i$, $i,j=1,2$ satisfy
\begin{align*}
	|a_{11}-1+\nu h|+|a_{12}+h|+|a_{21}-h|+|a_{22}-1|=\mcal O(h^2),\quad |b_1-1|+|b_2|=\mcal O(h).
\end{align*}

\textbf{Assumption 2} is given to ensure that \eqref{Mthd} for \eqref{Lan} with $V(q)=\frac{1}{2}q^2$ has at least first order convergence in the mean-square sense. 
This can be proved based on the fundamental theorem 
by comparing the one-step approximation of \eqref{Mthd} and Euler-Maruyama method. 
We will prove that under \textbf{Assumption 2}, \eqref{Mthd} can asymptotically preserve the LDP of $\{\mu_{\nu,\epsilon}\}_{\epsilon>0}$. 

\subsection{Existence and uniqueness of invariant measure}

The objective of this part is to show that \eqref{Mthd} has a unique invariant measure, which  is realized by deriving the general formulas of the numerical solutions $\{(P_n,Q_n)\}_{n\geq0}$ and their stable distributions. 

\begin{lem}\label{lem4.2}
	Under \textbf{Assumption 2}, it holds that
	\begin{itemize}
		\item[(1)] $(tr(A))^2-4det(A)=(\nu^2-4)h^2+\mcal O(h^3)$;
		\item[(2)] For all sufficiently small step-size $h$,  $|tr(A)|<1+det(A)<2$ or equivalently $|\lambda_i|<1$, $i=1,2$, where $\lambda_1$ and $\lambda_2$ are the eigenvalues of $A$. 
	\end{itemize}
Here $tr(A)$ and $det(A)$ denote the trace and determinant of $A$ respectively.
\end{lem}
\begin{proof}
(1)	Using \textbf{Assumption 2},
	\begin{gather}\label{trA}
		tr(A)=a_{11}+a_{22}=1-\nu h+\mcal O(h^2)+1+\mcal O(h^2)=2-\nu h+\mcal O(h^2).
	\end{gather}
Further, we have
\begin{align}\label{sec4k2}
&1+det(A)-tr(A)
=(1-a_{11})(1-a_{22})-a_{12}a_{21} 
=h^2+\mcal O(h^3).
\end{align}
By \eqref{trA} and \eqref{sec4k2},
\begin{align*}
	(tr(A))^2-4det(A)=&(tr(A)-2)^2-4(1+det(A)-tr(A)) \nonumber\\=&(-\nu h+\mcal O(h^2))^2-4(h^2+\mcal O(h^3)) 
	=(\nu^2-4)h^2+\mcal O(h^3).
\end{align*}

(2) By \textbf{Assumption 2}, we get after a calculation that
	\begin{align}\label{detA}
det(A)=&a_{11}a_{22}-a_{12}a_{21}
=1-\nu h+\mcal O(h^2),
\end{align}
which implies that for  sufficiently small $h$, $1+det(A)<2$.
It follows from \eqref{trA} and \eqref{detA} that
\begin{align*}
	tr(A)+1+det(A)=4-2\nu h+\mcal O(h^2),
\end{align*}
and thereby
$tr(A)>-(1+det(A))$ for  $h$ small enough. By \eqref{sec4k2}, $tr(A)<1+det(A)$ if $h$ is small enough. Thus, we have
$|tr(A)|<1+det(A)<2$ for all sufficiently small $h$.

Using the facts that $\lambda_{1,2}$ are the roots of $det(\lambda I_2-A)=\lambda^2-tr(A)\lambda+det(A)=0$ and that the moduli of both of the two roots (may be complex-valued) of equation $x^2-cx-d=0$, $c,d\in\mbb R$, are smaller than $1$ if and only if $|c|<1-d<2$, we complete the proof.
\end{proof}

\begin{theo}\label{tho4.3}
	Let  $\nu>0$ be fixed. If \textbf{Assumption 2} holds and $\nu\neq2$, then for all sufficiently small step-size $h>0$ and any $\epsilon>0$, the numerical solution $\{(P_n,Q_n)\}_{n\geq0}$ possesses a unique invariant measure $\mu_{\nu,\epsilon}^h=\mcal N(0,\Sigma)$ with
	\begin{align*}
		\Sigma=\frac{\epsilon h}{(\lambda_2-\lambda_1)^2}\begin{pmatrix}
		\Sigma_{11}&\Sigma_{12}\\\\
		\Sigma_{21}&\Sigma_{22}
		\end{pmatrix},
	\end{align*}
	where $\lambda_{1,2}=\frac{1}{2}\left(tr(A)\pm\sqrt{(tr(A))^2-4det(A)}\right)$ are the eigenvalues of $A$ and
	\begin{align}
		\Sigma_{11}:=&\frac{1}{1-\lambda_1^2}\Big(a_{12}b_2+b_1(a_{11}-\lambda_2)\Big)^2+\frac{1}{1-\lambda_2^2}\Big(a_{12}b_2+b_1(a_{11}-\lambda_1)\Big)^2\nonumber\\
		&+\frac{1}{1-\lambda_1\lambda_2}\Big(2a_{12}a_{21}b_1^2-2a_{12}^2b_2^2+2a_{12}b_1b_2(a_{22}-a_{11})\Big),\label{Sigma11}\\
		\Sigma_{22}:=&\frac{1}{1-\lambda_1^2}\Big(a_{21}b_1+b_2(\lambda_1-a_{11})\Big)^2+\frac{1}{1-\lambda_2^2}\Big(a_{21}b_1+b_2(\lambda_2-a_{11})\Big)^2\nonumber\\
		&+\frac{1}{1-\lambda_1\lambda_2}\Big(2a_{12}a_{21}b_2^2-2a_{21}^2b_1^2-2a_{21}b_1b_2(a_{22}-a_{11})\Big),\label{Sigma22}\\
		\Sigma_{12}=\Sigma_{21}:=&\frac{1}{1-\lambda_1^2}\Big(a_{21}(a_{11}-\lambda_2)b_1^2+a_{12}(\lambda_1-a_{11})b_2^2+2a_{12}a_{21}b_1b_2\Big)\nonumber\\
		&+\frac{1}{1-\lambda_2^2}\Big(a_{21}(a_{11}-\lambda_1)b_1^2+a_{12}(\lambda_2-a_{11})b_2^2+2a_{12}a_{21}b_1b_2\Big)\nonumber\\
		&+\frac{1}{1-\lambda_1\lambda_2}\Big(a_{22}-a_{11}\Big)\Big(a_{21}b_1^2-a_{12}b_2^2+b_1b_2(a_{22}-a_{11})\Big).\label{Sigma12}
	\end{align}
\end{theo}
\begin{proof}
	Solving the equation $det(\lambda I_2-A)=0$, we obtain 
	 $\lambda_{1,2}=\frac{1}{2}\left(tr(A)\pm\sqrt{(tr(A))^2-4det(A)}\right)$.
	By Lemma \ref{lem4.2}(1), $(tr(A))^2-4det(A)\sim (\nu^2-4)h^2$. If $\nu\neq 2$, for all sufficiently  small $h$, $(tr(A))^2-4det(A)\neq 0$, which yields that  $\lambda_1\neq \lambda_2$ and hence $A$ is diagonalizable. 		By a standard computation, the components of $A^n=(A_{ij}(n)),\,n=0,1,\ldots$, are given by 
	\begin{gather*}
	A_{11}(n)=\frac{1}{\lambda_2-\lambda_1}\big(a_{11}(\lambda_2^n-\lambda_1^n)+(\lambda_1^n\lambda_2-\lambda_1\lambda_2^n)\big),\nonumber\\
	A_{12}(n)=\frac{a_{12}(\lambda_2^n-\lambda_1^n)}{\lambda_2-\lambda_1},\qquad A_{21}(n)=\frac{a_{21}(\lambda_2^n-\lambda_1^n)}{\lambda_2-\lambda_1},\\
	A_{22}(n)=\frac{1}{\lambda_2-\lambda_1}\big(\lambda_2^{n+1}-\lambda_1^{n+1}+a_{11}(\lambda_1^n-\lambda_2^n)\big).
	\end{gather*}
From \eqref{sec4recur}, we obtain the general formula of $\{(P_n,Q_n)\}_{n\geq1}$ as follows
\begin{align*}
	P_n=&A_{11}(n)p+A_{12}(n)q+\sqrt{\epsilon}\sum_{j=0}^{n-1}\left[A_{11}(n-1-j)b_1+A_{12}(n-j-1)b_2\right]\Delta W_j,\\
	Q_n=&A_{21}(n)p+A_{22}(n)q+\sqrt{\epsilon}\sum_{j=0}^{n-1}\left[A_{21}(n-1-j)b_1+A_{22}(n-j-1)b_2\right]\Delta W_j.
\end{align*}
Due to Lemma \ref{lem4.2}(2), $|\lambda_i|<1$, $i=1,2$, for $h$ small enough. By the definitions of  $A_{ij}$, $i,j=1,2$,
\begin{align*}\lim\limits_{n\to\infty}\mbf E P_n=\lim\limits_{n\to\infty}A_{11}(n)p+\lim\limits_{n\to\infty}A_{12}(n)q=0,\\ \lim\limits_{n\to\infty}\mbf EQ_n=\lim\limits_{n\to\infty}A_{21}(n)p+\lim\limits_{n\to\infty}A_{22}(n)q=0 .
\end{align*}

Notice that
\begin{align}
	\mbf {Var}(P_n)=&\epsilon h\sum_{j=0}^{n-1}\left[b_1^2A_{11}^2(j)+b_2^2A_{12}^2(j)+2b_1b_2A_{11}(j)A_{12}(j)\right], \label{VarP}\\
	\mbf {Var}(Q_n)=&\epsilon h\sum_{j=0}^{n-1}\left[b_1^2A_{21}^2(j)+b_2^2A_{22}^2(j)+2b_1b_2A_{21}(j)A_{22}(j)\right],\label{VarQ}\\
	\mbf{Cov}(P_n,Q_n)=&\epsilon h\sum_{j=0}^{n-1}\left[b_1^2A_{11}(j)A_{21}(j)+b_2^2A_{12}(j)A_{22}(j)+b_1b_2\big(A_{11}(j)A_{22}(j)+A_{12}(j)A_{21}(j)\big)\right].\label{CovPQ}
\end{align}
Next, we compute the limits $\lim\limits_{n\to\infty}\mbf{Var}(P_n)$, $\lim_{n\to\infty}\limits\mbf{Var}(Q_n)$ and $\lim_{n\to\infty}\limits\mbf{Cov}(P_n,Q_n)$. 
In fact, 
\begin{align*}
	\lim_{n\to\infty}\sum_{j=0}^{n-1}A_{11}^2(j)=&\lim_{n\to\infty}\sum_{j=0}^{n-1}\frac{1}{(\lam)^2}\left[(\lambda_2-a_{11})^2\lambda_1^{2j}+(\lambda_1-a_{11})^2\lambda_2^{2j}+2(\lambda_2-a_{11})(a_{11}-\lambda_1)(\lambda_1\lambda_2)^j\right]\\
	=&\frac{1}{(\lam)^2}\left[\frac{(\lambda_2-a_{11})^2}{1-\lambda_1^2}+\frac{(\lambda_1-a_{11})^2}{1-\lambda_2^2}+\frac{2(\lambda_2-a_{11})(a_{11}-\lambda_1)}{1-\lambda_1\lambda_2}\right].
\end{align*}
Further, we have 
\begin{align*}
	\lim_{n\to\infty}\sum_{j=0}^{n-1}A_{12}^2(j)&=\frac{a_{12}^2}{(\lam)^2}\left(\frac{1}{1-\lambda_1^2}+\frac{1}{1-\lambda_2^2}-\frac{2}{1-\lambda_1\lambda_2}\right), \\
	\lim_{n\to\infty}\sum_{j=0}^{n-1}A_{11}(j)A_{12}(j)&=\frac{a_{12}}{(\lam)^2}\left(\frac{a_{11}-\lambda_2}{1-\lambda_1^2}+\frac{a_{11}-\lambda_1}{1-\lambda_2^2}+\frac{\lambda_1+\lambda_2-2a_{11}}{1-\lambda_1\lambda_2}\right).
\end{align*}
By the above formulas and \eqref{VarP},
\begin{align}
	\lim_{n\to\infty}\mbf {Var}(P_n)=&\epsilon h\left[\frac{b_1^2}{(\lam)^2}\left(\frac{(\lambda_2-a_{11})^2}{1-\lambda_1^2}+\frac{(\lambda_1-a_{11})^2}{1-\lambda_2^2}+\frac{2(\lambda_2-a_{11})(a_{11}-\lambda_1)}{1-\lambda_1\lambda_2}\right)\right.\nonumber\\
	&+\frac{a_{12}^2b_2^2}{(\lam)^2}\left(\frac{1}{1-\lambda_1^2}+\frac{1}{1-\lambda_2^2}-\frac{2}{1-\lambda_1\lambda_2}\right)\nonumber\\
	&\left.+\frac{2a_{12}b_1b_2}{(\lam)^2}\left(\frac{a_{11}-\lambda_2}{1-\lambda_1^2}+\frac{a_{11}-\lambda_1}{1-\lambda_2^2}+\frac{\lambda_1+\lambda_2-2a_{11}}{1-\lambda_1\lambda_2}\right)\right].\label{sec4k3}
\end{align}
Noting that $\lambda_1\lambda_2=det(A)=a_{11}a_{22}-a_{12}a_{21}$ and $\lambda_1+\lambda_2=tr(A)=a_{11}+a_{22}$, we have 
\begin{align}\label{sec4k4}
(\lambda_2-a_{11})(a_{11}-\lambda_1)=a_{12}a_{21},
\qquad \lambda_1+\lambda_2-2a_{11}=a_{22}-a_{11}.
\end{align}
Substituting \eqref{sec4k4} into \eqref{sec4k3} and by rearranging, we obtain $\lim_{n\to\infty}\mbf {Var}(P_n)=\frac{\epsilon h\Sigma_{11}}{(\lam)^2}$.

Similarly, by  $|\lambda_{1,2}|<1$, \eqref{sec4k4} and the fact $(\lambda_1-a_{11})^2+(\lambda_2-a_{11})^2=2a_{12}a_{21}+(a_{11}-a_{22})^2$,
\begin{gather*}
	\lim_{n\to\infty}\sum_{j=0}^{n-1}A_{21}^2(j)=\frac{a_{21}^2}{(\lam)^2}\left(\frac{1}{1-\lambda_1^2}+\frac{1}{1-\lambda_2^2}-\frac{2}{1-\lambda_1\lambda_2}\right), \\
	\lim_{n\to\infty}\sum_{j=0}^{n-1}A_{22}^2(j)=\frac{1}{(\lam)^2}\left(\frac{(a_{11}-\lambda_1)^2}{1-\lambda_1^2}+\frac{(\lambda_2-a_{11})^2}{1-\lambda_2^2}+\frac{2a_{12}a_{21}}{1-\lambda_1\lambda_2}\right),\\
	\lim_{n\to\infty}\sum_{j=0}^{n-1}A_{21}(j)A_{22}(j)=\frac{a_{21}}{(\lam)^2}\left(\frac{\lambda_1-a_{11}}{1-\lambda_1^2}+\frac{\lambda_2-a_{11}}{1-\lambda_2^2}-\frac{a_{22}-a_{11}}{1-\lambda_1\lambda_2}\right),\\
	\lim_{n\to\infty}\sum_{j=0}^{n-1}A_{11}(j)A_{21}(j)=\frac{a_{21}}{(\lam)^2}\left(\frac{a_{11}-\lambda_2}{1-\lambda_1^2}+\frac{a_{11}-\lambda_1}{1-\lambda_2^2}+\frac{a_{22}-a_{11}}{1-\lambda_1\lambda_2}\right),\\
	\lim_{n\to\infty}\sum_{j=0}^{n-1}A_{12}(j)A_{22}(j)=\frac{a_{12}}{(\lam)^2}\left(\frac{\lambda_1-a_{11}}{1-\lambda_1^2}+\frac{\lambda_2-a_{11}}{1-\lambda_2^2}-\frac{a_{22}-a_{11}}{1-\lambda_1\lambda_2}\right),\\
	\lim_{n\to\infty}\sum_{j=0}^{n-1}A_{11}(j)A_{22}(j)=\frac{1}{(\lam)^2}\left(\frac{a_{12}a_{21}}{1-\lambda_1^2}+\frac{a_{12}a_{21}}{1-\lambda_2^2}+\frac{2a_{12}a_{21}+(a_{11}-a_{22})^2}{1-\lambda_1\lambda_2}\right),\\
	\lim_{n\to\infty}\sum_{j=0}^{n-1}A_{12}(j)A_{21}(j)=\frac{a_{12}a_{21}}{(\lam)^2}\left(\frac{1}{1-\lambda_1^2}+\frac{1}{1-\lambda_2^2}-\frac{2}{1-\lambda_1\lambda_2}\right).
\end{gather*}
Substituting the above formulas into \eqref{VarQ} and \eqref{CovPQ} yields
$\lim_{n\to\infty}\mbf {Var}(Q_n)=\frac{\epsilon h\Sigma_{22}}{(\lam)^2}$ and $\lim_{n\to\infty}\mbf{Cov}(P_n,Q_n)=\frac{\epsilon h\Sigma_{12}}{(\lam)^2}$.
Therefore,
 for any $(p,q)\in\mbb R^2$, 
 the law of $(P_n,Q_n)$ weakly converges to $\mu_{\nu,\epsilon}^h=\mcal N(0,\Sigma)$ as $n$ tends to $\infty$. For any $n\ge1$ and $\varphi\in\mbf B_b(\mbb R^2)$, define $\mscr P_{n}\varphi(p,q)=\mbf E[\varphi(P_n,Q_n)]$, where $(P_0,Q_0)=(p,q)$. 
 Then $\mscr P_{n}$ is \emph{Feller} because $\{(P_n,Q_n)\}_{n\geq 0}$ admits a smooth transition density. 
 Finally,
 it follows from Proposition \ref{pro2.8} that the numerical solution $\{(P_n,Q_n)\}_{n\geq 0}$ possesses the unique invariant measure $\mu_{\nu,\epsilon}^h$. 
 \end{proof}

\begin{rem}\label{rem4.4}
As is shown in the proof of Theorem \ref{tho4.3}, as long as  $A$ has two different eigenvalues $\lambda_{1,2}$ with $|\lambda_{1,2}|<1$, the numerical method \eqref{Mthd} has a unique invariant measure given by $\mcal N(0,\Sigma)$. 
\end{rem}

\subsection{Asymptotically preserving the LDP of $\{\mu_{\nu,\epsilon}\}_{\epsilon>0}$}
In this part,
based on Theorem \ref{tho4.3}, 
 we derive the LDP of $\{\mu_{\nu,\epsilon}^h\}_{\epsilon>0}$ by utilizing G\"artner--Ellis theorem.
Further, the numerical method \eqref{Mthd}
is shown to asymptotically preserve the LDP of $\{\mu_{\nu,\epsilon}\}_{\epsilon>0}$. For preparation, we give the following lemma.

\begin{lem}\label{lem4.4}
	Let \textbf{Assumption 2} hold and $\nu\neq 2$. Then we have
$\Sigma_{11}\sim \frac{\nu^2-4}{2\nu}h$,  $\Sigma_{22}\sim \frac{\nu^2-4}{2\nu}h$ and $\lim_{h\to0}\Sigma_{12}/h=\lim_{h\to0}\Sigma_{21}/h=0$.
\end{lem}
\begin{proof}
	First we consider the case $\nu>2$.  In this case, by Lemma \ref{lem4.2}(1), for all sufficiently small $h$, $A$ has two real-valued eigenvalues $\lambda_{1,2}=\frac{1}{2}\left(tr(A)\pm\sqrt{(tr(A))^2-4det(A)}\right)$. It follows from \ref{lem4.2}(1) and \eqref{trA} that
	\begin{gather}
		\lambda_1=\frac{1}{2}\left(2-\nu h+\mcal O(h^2)+\sqrt{\nu^2-4}\,h+\mcal O(h^{3/2})\right)=1-\frac{\nu-\sqrt{\nu^2-4}}{2}h+\mcal O(h^{3/2}). \label{lam1}\\
		\lambda_2=\frac{1}{2}\left(2-\nu h+\mcal O(h^2)-\sqrt{\nu^2-4}\,h+\mcal O(h^{3/2})\right)=1-\frac{\nu+\sqrt{\nu^2-4}}{2}h+\mcal O(h^{3/2}), \label{lam2}
	\end{gather}
	which leads to
	\begin{gather*}
		\lambda_1^2=1-\left(\nu-\sqrt{\nu^2-4}\right)h+\mcal O(h^{3/2}),\qquad	\lambda_2^2=1-\left(\nu+\sqrt{\nu^2-4}\right)h+\mcal O(h^{3/2}).
	\end{gather*}
Thus, it holds that
\begin{align}
	(1-\lambda_1^2)\sim \left(\nu-\sqrt{\nu^2-4}\right)h,\qquad (1-\lambda_2^2)\sim \left(\nu+\sqrt{\nu^2-4}\right)h. \label{1-lam^2}
\end{align}
According to \textbf{Assumption 2} and \eqref{lam2}, $$a_{11}-\lambda_2=1-\nu h+\mcal O(h^2)-1+\frac{\nu+\sqrt{\nu^2-4}}{2}h+\mcal O(h^{3/2})=\frac{\sqrt{\nu^2-4}-\nu}{2}h+\mcal O(h^{3/2}).$$
Hence, 
\begin{align*}
	&\left(a_{12}b_2+b_1(a_{11}-\lambda_2)\right)^2=a_{12}^2b_2^2+b_1^2(a_{11}-\lambda_2)^2+2a_{12}b_1b_2(a_{11}-\lambda_2) \\
	=&\mcal O(h^4)+(1+\mcal O(h))^2\left(\frac{1}{4}\left(\nu-\sqrt{\nu^2-4}\right)^2h^2+\mcal O(h^{5/2})\right)+\mcal O(h^{3})\\
	=&\frac{1}{4}\left(\nu-\sqrt{\nu^2-4}\right)^2h^2+\mcal O(h^{5/2}).
\end{align*}
That is
	$\left(a_{12}b_2+b_1(a_{11}-\lambda_2)\right)^2\sim \frac{1}{4}(\nu-\sqrt{\nu^2-4})^2h^2.$
Similarly, one has
	$\left(a_{12}b_2+b_1(a_{11}-\lambda_1)\right)^2\sim \frac{1}{4}(\nu+\sqrt{\nu^2-4})^2h^2.$
Due to \textbf{Assumption 2} and \eqref{detA},
\begin{gather}
	1-\lambda_1\lambda_2=1-det(A)=\nu h+\mcal O(h^2)\sim \nu h, \label{sec4k7} \\
	2a_{12}a_{21}b_1^2-2a_{12}^2b_2^2+2a_{12}b_1b_2(a_{22}-a_{11})=\left(-2h^2+\mcal O(h^3)\right)+\mcal O(h^4)+\mcal O(h^3)\sim -2h^2. \label{sec4k8}
\end{gather}
It follows from \eqref{1-lam^2}-\eqref{sec4k8} and the definition of $\Sigma_{11}$ that
\begin{align*}
	\Sigma_{11}=&\frac{1}{1-\lambda_1^2}\Big(a_{12}b_2+b_1(a_{11}-\lambda_2)\Big)^2+\frac{1}{1-\lambda_2^2}\Big(a_{12}b_2+b_1(a_{11}-\lambda_1)\Big)^2\\
	&+\frac{1}{1-\lambda_1\lambda_2}\Big(2a_{12}a_{21}b_1^2-2a_{12}^2b_2^2+2a_{12}b_1b_2(a_{22}-a_{11})\Big)\\
	\sim&\frac{\left(\nu-\sqrt{\nu^2-4}\right)^2h^2}{4\left(\nu-\sqrt{\nu^2-4}\right)h}+\frac{\left(\nu+\sqrt{\nu^2-4}\right)^2h^2}{4\left(\nu+\sqrt{\nu^2-4}\right)h}+\frac{-2h^2}{\nu h}\\
	\sim& (\frac{\nu}{2}-\frac{2}{\nu})h=\frac{\nu^2-4}{2\nu}h.
\end{align*}

Also, one can prove that
\begin{gather}
	\left(a_{21}b_1+b_2(\lambda_1-a_{11})\right)^2\sim h^2,\qquad \left(a_{21}b_1+b_2(\lambda_2-a_{11})\right)^2\sim h^2,\label{sec4k9}\\
	2a_{12}a_{21}b_2^2-2a_{21}^2b_1^2-2a_{21}b_1b_2(a_{22}-a_{11})\sim -2h^2.\label{sec4k10}
	\end{gather}
Combining \eqref{1-lam^2}, \eqref{sec4k7}, \eqref{sec4k9} and \eqref{sec4k10}, we obtain
\begin{align*}
	\Sigma_{22}=&\frac{1}{1-\lambda_1^2}\Big(a_{21}b_1+b_2(\lambda_1-a_{11})\Big)^2+\frac{1}{1-\lambda_2^2}\Big(a_{21}b_1+b_2(\lambda_2-a_{11})\Big)^2\\
	&+\frac{1}{1-\lambda_1\lambda_2}\Big(2a_{12}a_{21}b_2^2-2a_{21}^2b_1^2-2a_{21}b_1b_2(a_{22}-a_{11})\Big)\\
	\sim&\frac{h^2}{\left(\nu-\sqrt{\nu^2-4}\right)h}+\frac{h^2}{\left(\nu+\sqrt{\nu^2-4}\right)h}+\frac{-2h^2}{\nu h}\\
	\sim&\frac{\nu^2-4}{2\nu}h.
\end{align*}

It remains to estimate $\Sigma_{12}$. Again by \eqref{lam1}, \eqref{lam2} and \textbf{Assumption 2},
\begin{align}
	S_1:&=a_{21}(a_{11}-\lambda_2)b_1^2+a_{12}(\lambda_1-a_{11})b_2^2+2a_{12}a_{21}b_1b_2=\frac{\sqrt{\nu^2-4}-\nu}{2}h^2+\mcal O(h^3),\label{S1}\\
	S_2:&=a_{21}(a_{11}-\lambda_1)b_1^2+a_{12}(\lambda_2-a_{11})b_2^2+2a_{12}a_{21}b_1b_2=-\frac{\sqrt{\nu^2-4}+\nu}{2}h^2+\mcal O(h^3),\label{S2}\\
	S_3:&=(a_{22}-a_{11})\left(a_{21}b_1^2-a_{12}b_2^2+b_1b_2(a_{22}-a_{11})\right)=\nu h^2+\mcal O(h^3)\label{S3}.
\end{align}
By \eqref{1-lam^2} and \eqref{sec4k7}, we have
\begin{gather*}
	\lim_{h\to0}\frac{S_1}{(1-\lambda_1^2)h}=\lim_{h\to0}\frac{\left(\sqrt{\nu^2-4}-\nu\right)h^2/2}{\left(\nu-\sqrt{\nu^2-4}\right)h^2}=-\frac{1}{2},\\
	\lim_{h\to0}\frac{S_2}{(1-\lambda_2^2)h}=\lim_{h\to0}\frac{-\left(\sqrt{\nu^2-4}+\nu\right)h^2/2}{\left(\nu+\sqrt{\nu^2-4}\right)h^2}=-\frac{1}{2},\\
	\lim_{h\to0}\frac{S_3}{\left(1-\lambda_1\lambda_2\right)h}=\frac{\nu h^2}{\nu h^2}=1.
\end{gather*}
Thus, by the above formulas,
\begin{align*}
	\lim_{h\to0}\frac{\Sigma_{12}}{h}=\lim_{h\to0}\frac{\Sigma_{21}}{h}=\lim_{h\to0}\left(\frac{S_1}{(1-\lambda_1^2)h}+\frac{S_2}{(1-\lambda_2^2)h}+\frac{S_3}{\left(1-\lambda_1\lambda_2\right)h}\right)=0.
\end{align*}

Finally, if $\nu<2$, then $\lambda_{1,2}$ are complex numbers for all sufficiently small $h>0$. In this case, repeating the above procedure just by replacing $\sqrt{\nu^2-4}$ by $\bm{i}\sqrt{4-\nu^2}$ with $\bm{i}^2=-1$, one can show that the conclusions hold.
This finishes the proof.
\end{proof}

Now we prove the main result of this section.
\begin{theo}\label{tho4.4}
	For the numerical method \eqref{Mthd} approximating \eqref{Lan} with $V(q)=\frac{1}{2}q^2$, let \textbf{Assumption 2} hold and $\nu\neq2$ be fixed. Then
	\begin{itemize}
		\item[(1)] For all sufficiently small $h>0$, the invariant measure $\{\mu_{\nu,\epsilon}^h\}_{\epsilon>0}$ of \eqref{Mthd} satisfies an LDP with the good rate function $I^h$ given by
		\begin{align}\label{I^h}
		I^h(p,q):=\frac{\left((tr(A))^2-4det(A)\right)\left(\Sigma_{22}p^2+\Sigma_{11}q^2-2\Sigma_{12}pq\right)}{2h\left(\Sigma_{11}\Sigma_{22}-\Sigma_{12}^2\right)}.
		\end{align}
 
 \item[(2)] The numerical method \eqref{Mthd} asymptotically preserves the LDP of $\{\mu_{\nu,\epsilon}\}_{\epsilon>0}$, i.e.,
 \begin{align*}
 	\lim_{\epsilon\to0}I^h(p,q)=\nu(p^2+q^2),\qquad\forall \quad(p,q)\in\mbb R^2.
 \end{align*} 
	\end{itemize}
\end{theo}
\begin{proof}
	(1) We divide this proof into three steps.
	
	\textit{Step 1: We give the logarithmic moment generating function of $\{\mu_{\nu,\epsilon}^h\}_{\epsilon>0}$.}
		
	Let $X^\epsilon$ obey to the distribution $\mu_{\nu,\epsilon}^h=\mcal N(0,\Sigma)$, where $\Sigma=\frac{\epsilon h}{(\lam)^2}\begin{pmatrix}
	\Sigma_{11}&\Sigma_{12}\\
	\Sigma_{21}&\Sigma_{22}
	\end{pmatrix}$.
	For any $\theta=(\theta_1,\theta_2)^{\top}\in\mbb R^2$, the logarithmic moment generating function of $\{X^\epsilon\}_{\epsilon>0}$ is
	\begin{align*}
		\Lambda^h(\theta)=&\lim_{\epsilon\to0}\epsilon\ln \mbf Ee^{\langle X^\epsilon,\,\theta\rangle/\epsilon}=\lim_{\epsilon\to0}\epsilon\left[\frac{1}{2\epsilon^2}\mbf {Var}\langle X^\epsilon,\,\theta\rangle\right]=\lim_{\epsilon\to0}\frac{1}{2\epsilon}\theta^\top\Sigma\theta.
	\end{align*}
Since $\Sigma_{ij}$, $i,j=1,2$ are independent of $\epsilon$,
\begin{align}\label{Lam^h}
	\Lambda^h(\theta)=\frac{h\left(\Sigma_{11}\theta_1^2+\Sigma_{22}\theta_2^2+2\Sigma_{12}\theta_1\theta_2\right)}{2(\lambda_2-\lambda_1)^2}=\frac{h\left(\Sigma_{11}\theta_1^2+\Sigma_{22}\theta_2^2+2\Sigma_{12}\theta_1\theta_2\right)}{2\left((tr(A))^2-4det(A)\right)},
\end{align}
where we used the fact $(\lam)^2=(\lambda_1+\lambda_2)^2-4\lambda_1\lambda_2=(tr(A))^2-4det(A)$.

\vspace{2mm}
\textit{Step 2: We show that $\{\mu_{\nu,\epsilon}^h\}_{\epsilon>0}$ is exponentially smooth based on the finiteness of $\Lambda^h$.}

Let $e_1=(1,0)^\top$ and $e_2=(0,1)^\top$. Then, by \eqref{Lam^h},
\begin{align}\label{sec4finite}
	\Lambda^h(\pm e_1),\Lambda^h(\pm e_2)<+\infty.
\end{align}
For any $L>0$, by Markov's inequality,
\begin{align*}
	\mbf P\big(\langle X^\epsilon,\,e_1 \rangle>L/2\big)=\mbf P\big(e^{\langle X^\epsilon,\,e_1 \rangle/\epsilon}>e^{ L/(2\epsilon)}\big)\leq e^{-L/(2\epsilon)}\mbf Ee^{\langle X^\epsilon,\,e_1 \rangle/\epsilon}.
\end{align*}
According to the definition of $\Lambda^h$,
\begin{align*}
	\limsup_{\epsilon\to 0}\epsilon\ln \mbf P\big(\langle X^\epsilon,\,e_1 \rangle>L/2\big)\leq -\frac{L}{2}+\lim_{\epsilon\to\infty}\epsilon\ln\mbf Ee^{\langle X^\epsilon,\,e_1 \rangle/\epsilon}=-\frac{L}{2}+\Lambda^h(e_1).
\end{align*}
Similary, by Markov's inequality,
\begin{align*}
	&\limsup_{\epsilon\to0}\epsilon\ln \mbf P\big(\langle X^\epsilon,\,e_1 \rangle<-L/2\big)=\limsup_{\epsilon\to0}\epsilon\ln \mbf P\big(\langle X^\epsilon,\,-e_1 \rangle>L/2\big)
=-\frac{L}{2}+\Lambda^h(-e_1).
\end{align*}
Notice that $\mbf P\left(\left|\langle X^\epsilon,\,e_1 \rangle\right|>L/2\right)=\mbf P\left(\langle X^\epsilon,\,e_1 \rangle>L/2\right)+\mbf P\left(\langle X^\epsilon,\,e_1 \rangle<-L/2\right)$. It follows from Proposition \ref{limsup} that
\begin{align}\label{sec4k11}
	&\limsup_{\epsilon\to0}\epsilon\ln \mbf P\left(\left|\langle X^\epsilon,\,e_1 \rangle\right|>L/2\right)\nonumber\\
	=&\max\left\{\limsup_{\epsilon\to0}\epsilon\ln \mbf P\big(\langle X^\epsilon,\,e_1 \rangle>L/2\big),\limsup_{\epsilon\to0}\epsilon\ln \mbf P\big(\langle X^\epsilon,\,e_1 \rangle<-L/2\big)\right\}\nonumber\\
	\leq&\max\left\{-\frac{L}{2}+\Lambda^h(e_1),\,-\frac{L}{2}+\Lambda^h(-e_1)\right\}.
\end{align}
Similarly, one can prove
\begin{align}\label{sec4k12}
	\limsup_{\epsilon\to0}\epsilon\ln \mbf P\big(\left|\langle X^\epsilon,\,e_2 \rangle\right|>L/2\big)\leq\max\left\{-\frac{L}{2}+\Lambda^h(e_2),-\frac{L}{2}+\Lambda^h(-e_2)\right\}.
\end{align}
Hence,  combining  Proposition \ref{limsup}, \eqref{sec4k11} and \eqref{sec4k12} leads to
\begin{align}
	&\limsup_{\epsilon\to0}\epsilon\ln\mbf P\big(|X^\epsilon|>L\big)\nonumber\\
	\leq&\limsup_{\epsilon\to0}\epsilon\ln\Big[\mbf P\left(\left|\langle X^\epsilon,\,e_1 \rangle\right|>L/2\right)+\mbf P\left(\left|\langle X^\epsilon,\,e_2 \rangle\right|>L/2\right)\Big]\nonumber\\
	= &\max\left\{\limsup_{\epsilon\to0}\epsilon\ln\mbf P\left(\left|\langle X^\epsilon,\,e_1 \rangle\right|>L/2\right),\,\limsup_{\epsilon\to0}\epsilon\ln\mbf P\left(\left|\langle X^\epsilon,\,e_2 \rangle\right|>L/2\right)\right\}\nonumber\\
	\leq &-\frac{L}{2}+\max\left\{
	\Lambda^h(\pm e_1),	\Lambda^h(\pm e_2)\right\}.\label{sec4k13}
\end{align}
Denote $K_L=\bar B(0,L)$, which is the compact subset of $\mbb R^2$.
Using \eqref{sec4finite} and \eqref{sec4k13} gives
\begin{align*}
	\lim_{L\to\infty}\lim_{\epsilon\to 0}\epsilon\ln \mbf P\left(X^\epsilon\in K_L^c\right)=\lim_{L\to\infty}\lim_{\epsilon\to 0}\epsilon\ln \mbf P\left(|X^\epsilon|>L\right)=-\infty.
\end{align*}
Thus, the exponential tightness of $\{X^\epsilon\}_{\epsilon>0}$ or $\{\mu_{\nu,\epsilon}^h\}_{\epsilon}$ follows from Definition \ref{exptight}.

\vspace{2mm}
\textit{Step 3: We give the explicit expression of the rate function of $\{\mu_{\nu,\epsilon}^h\}_{\epsilon>0}$.}

Notice that $\Lambda^h$ is finite valued and Gateaux differentiable.
By Theorem \ref{GE}, for  sufficiently small $h$, $\{\mu_{\nu,\epsilon}^h\}_{\epsilon>0}$ satisfies an LDP with the good rate function $I^h(x)=\sup_{\theta\in\mbb R^2}\left\{\langle x,\,\theta\rangle-\Lambda^h(\theta)\right\}$ for any $x\in\mbb R^2$. 
 Introduce
\begin{align*}
		M=\frac{ h}{2\left((tr(A))^2-4det(A)\right)}\begin{pmatrix}
		\Sigma_{11}&\Sigma_{12}\\\\
		\Sigma_{21}&\Sigma_{22}
		\end{pmatrix}.
	\end{align*}
By Lemmas \ref{lem4.2} and \ref{lem4.4},  for sufficiently small $h$,
  $\Sigma_{11}\Sigma_{22}-\Sigma_{12}\Sigma_{21}>0$ and hence $det(M)>0$.
Noticing that $M=\frac{1}{2\epsilon}\Sigma$ and $\Sigma$ is non-negative definite, we have that 
$M$ is positive definite provided that $h$ is sufficiently small.
Then for $x=(p,q)^\top$,
\begin{align*}
 I^h(x)&=\sup_{\theta\in\mbb R^2}\left\{\langle x,\,\theta\rangle-\theta^{\top}M\theta\right\}=\sup_{\theta\in\mbb R^2}\left\{\langle M^{-\frac{1}{2}}x,\,M^{\frac{1}{2}}\theta\rangle-|M^{\frac{1}{2}}\theta|^2\right\}=\frac{1}{4}| M^{-\frac{1}{2}}x|^2\\
&=\frac{1}{4}x^{\top}M^{-1}x= \frac{\left((tr(A))^2-4det(A)\right)\left(\Sigma_{22}p^2+\Sigma_{11}q^2-2\Sigma_{12}pq\right)}{2h\left(\Sigma_{11}\Sigma_{22}-\Sigma_{12}^2\right)}.
\end{align*}
This proves \eqref{I^h}. Finally,
by Lemmas \ref{lem4.2} and \ref{lem4.4}, we deduce that
\begin{align*}
	\lim_{h\to0}I^h(p,q)=&\lim_{h\to0}\frac{(tr(A))^2-4det(A)}{2h^2}\frac{\frac{\Sigma_{22}}{h}p^2+\frac{\Sigma_{11}}{h}q^2-2\frac{\Sigma_{12}}{h}pq}{\frac{\Sigma_{11}}{h}\frac{\Sigma_{22}}{h}-\frac{\Sigma_{12}^2}{h^2}}=\frac{\nu^2-4}{2}\frac{p^2+q^2}{\frac{\nu^2-4}{2\nu}}=\nu(p^2+q^2).
\end{align*}
\end{proof}

\section{LDPs of invariant measures of numerical methods with strong dissipation}\label{Sec5}
In last section, we prove that the numerical method \eqref{Mthd} asymptotically preserves the LDP of $\{\mu_{\nu,\epsilon}\}_{\epsilon>0}$. In this section, we study the preservation of numerical methods for the LDP of $\{\mu_{\nu,\epsilon}\}_{\nu>0}$. We still restrict our study to the linear case with $V(q)=\frac{1}{2}q^2$. We will show that the LDP of invariant measures $\{\mu_{\nu,\epsilon}^h\}_{\nu>0}$ of numerical method in the strong dissipation limit  is quite different from the LDP of $\{\mu_{\nu,\epsilon}^h\}_{\epsilon>0}$ in the small noise limit (Recall that $\mu_{\nu,\epsilon}^h=\mcal N(0,\Sigma)$). 
Since $A$ and $b$ in \eqref{Mthd} depend on the parameter $\nu$, the logarithmic moment generating functions of $\{\mu_{\nu,\epsilon}^h\}_{\nu>0}$ can not be explicitly given. In what follows, we study the LDP of $\{\mu_{\nu,\epsilon}^h\}_{\nu>0}$ of the stochastic $\theta$-method, $\theta\in[1/2,1]$. We find that  the midpoint scheme ($\theta=1/2$) can asymptotically preserve the LDP of $\{\mu_{\nu,\epsilon}\}_{\nu>0}$, while the stochastic $\theta$-method, $\theta\in(1/2,1]$ fails to do. This differs from the result of Section \ref{Sec4} where  we prove that any numerical method in the form of \eqref{Mthd} asymptotically preserves the LDP of $\{\mu_{\nu,\epsilon}\}_{\epsilon>0}$.


The stochastic $\theta$-method $(\frac{1}{2}\le\theta\le 1)$ for equation \eqref{Lan} with $V(q)=\frac{1}{2}q^2$ reads
\begin{align}\label{stheta}
\left(\begin{array}{c}
P_{n+1}\\
Q_{n+1}
\end{array}\right)=
A_\theta(h)\left(\begin{array}{cc}
P_n\\
Q_n
\end{array}\right)
+\sqrt{\epsilon}
b_\theta(h)\Delta W_{n}, \qquad n=0,1,2,\ldots
\end{align}
with
\begin{align*}
A_\theta(h):=\left(
\begin{array}{cc}
\frac{1-(\theta h^2+\nu h)(1-\theta)}{1+\triangle}&\frac{-h}{1+\triangle}\\
\\
\frac{h}{1+\triangle}&\frac{1-\theta(1-\theta)h^2+\nu \theta h}{1+\triangle}
\end{array}\right),\qquad b_\theta(h):=	\left(\begin{array}{cc}
\frac{1}{1+\triangle}\\\\
\frac{\theta h}{1+\triangle}
\end{array}\right) 
\end{align*}
and $\triangle=\theta(\theta h^2+\nu h)$.

\begin{lem}\label{lem5.2}
Let $\theta\in[1/2,1]$ and $\epsilon>0$ be fixed. For any  $h>0$ and  $\nu>2$, the stochastic $\theta$-method \eqref{stheta}  possesses a unique invariant measure $\mcal N(0,\Sigma_{\theta})$ with 
	\begin{align*}
\Sigma_{\theta}=\frac{\epsilon h}{(\lambda_2-\lambda_1)^2}\begin{pmatrix}
\Sigma_{11}&\Sigma_{12}\\\\
\Sigma_{21}&\Sigma_{22}
\end{pmatrix},
\end{align*} 
where $\lambda_{1,2}=\frac{2+\nu h(2\theta-1)-2\theta(1-\theta)h^2\pm\sqrt{\nu^2-4}h}{2(1+\triangle)}$
are the eigenvalues of $A_\theta$ and $\Sigma_{11},\Sigma_{22},\Sigma_{12}=\Sigma_{21}$ are given according to  \eqref{Sigma11}, \eqref{Sigma22} and \eqref{Sigma12}, respectively.
\end{lem}
\begin{proof}
	First, we have
	\begin{gather*}
	tr(A_\theta)=\frac{2+\nu h(2\theta-1)-2\theta(1-\theta)h^2}{1+\triangle},\quad det(A_\theta)=\frac{1-\nu h(1-\theta)+(1-\theta)^2h^2}{1+\triangle}.
	\end{gather*}
Hence, we have that for any $h,\nu>0$,
\begin{gather*}
	1-det(A_{\theta})=\frac{(2\theta-1)h^2+\nu h}{1+\Delta}>0,\\
	1+det(A_{\theta})-tr(A_{\theta})=\frac{h^2}{1+\Delta}>0,\\
	1+det(A_{\theta})+tr(A_{\theta})=\frac{(2\theta-1)^2h^2+2\nu h(2\theta-1)+4}{1+\Delta}>0.
\end{gather*}
These imply $|tr(A_\theta)|<1+det(A_\theta)<2$. By Lemma \ref{lem4.2}, it follows that $|\lambda_{1,2}|<1$.
Next, it is verified that for any $h>0$ and $\nu>2$,
$\lambda_1\neq\lambda_2$.
Finally, by Remark \ref{rem4.4}, the stochastic $\theta$-method \eqref{stheta} admits a unique invariant measure given by $\mcal N(0,\Sigma_\theta)$ for any $h>0$ and $\nu>2$.
\end{proof}

\begin{rem}
	If $\theta\in[0,1/2)$, then $2\theta-1<0$. In this case, for any fixed $h>0$, there is no constant $\nu_0(h)$ such that for any $\nu>\nu_0(h)$, $1+det(A_{\theta})+tr(A_{\theta})>0$. In this case, $|\lambda_1|>1$ or $|\lambda_2|>1$, the existence of invariant measure of the method \eqref{stheta} can not be  ensured for  sufficiently large $\nu$.
\end{rem}

\begin{theo}\label{tho5.3}
	Let $\theta\in[1/2,1]$ be fixed. Then
	the stochastic $\theta$-method  \eqref{stheta} asymptotically preserves the LDP of $\{\mu_{\nu,\epsilon}\}_{\nu>0}$ if and only if $\theta=1/2.$
	
\end{theo}
\begin{proof}
Let $h>0$ be fixed and $\nu>2$. Then Lemma \ref{lem5.2} indicates that the method \eqref{stheta} admits the unique invariant measure $\mcal N(0,\Sigma_{\theta})$. Let   $Y_h^\nu$ obey the distribution $\mu_{\nu,\epsilon}^h=\mcal N(0,\Sigma_\theta)$. Introducing $a(\nu,\epsilon,h):=\frac{\epsilon \nu h}{2(\lambda_2-\lambda_1)^2}$, then for any $y=(y_1,y_2)^\top\in\mbb R^2$, the logarithmic moment generating function of $\{Y_h^\nu\}_{\nu>0}$ is
	\begin{align}\nonumber\label{Yhnu}
	&\Lambda^h(y)=\lim_{\nu\to\infty}\nu^{-1}\ln \mbf Ee^{\nu\langle Y_h^\nu,\,y\rangle}=\lim_{\nu\rightarrow\infty}\nu^{-1}\left[\frac{\nu^2}{2}\mbf {Var}\langle Y_h^\nu,\,y\rangle\right]=\lim_{\nu\rightarrow\infty}\frac{\nu}{2}y^\top\Sigma_{\theta} y\\
	=&\lim_{\nu\rightarrow\infty}a(\nu,\epsilon,h)\Sigma_{11}y_1^2+\lim_{\nu\rightarrow\infty}2a(\nu,\epsilon,h)\Sigma_{12}y_1y_2+\lim_{\nu\rightarrow\infty}a(\nu,\epsilon,h)\Sigma_{22}y_2^2.
	\end{align}
	Notice that 
	\begin{align}\label{apprnu}
	\lim_{\nu\rightarrow\infty}\frac{\sqrt{\nu^2-4}}{\nu}=1,~\text{and}~\lim_{\nu\rightarrow\infty}\frac{\nu-\sqrt{\nu^2-4}}{\frac{2}{\nu}}=1.
	\end{align}
	Now we proceed to calculate the above three limits in \eqref{Yhnu}. A direct computation leads to
	\begin{align}\label{avh}
	&a(\nu,\epsilon,h)=\frac{\epsilon\nu(1+\triangle)^2}{2(\nu^2-4)h},~~a_{11}-\lambda_{1,2}=\frac{-\nu\mp\sqrt{\nu^2-4}}{2(1+\triangle)}h,\\\label{1-lam}
	&1-\lambda_{1,2}=\frac{(\nu\mp\sqrt{\nu^2-4})h+2\theta h^2}{2(1+\triangle)},\\\label{1+lam}
	&1+\lambda_{1,2}=\frac{4+\nu h(4\theta-1)+2h^2\theta(2\theta-1)\pm\sqrt{\nu^2-4}h}{2(1+\triangle)}.
	\end{align}
	
	\textit{Step 1: We prove that for $\theta\in[\frac{1}{2},1]$,
		$\lim_{\nu\rightarrow\infty}a(\nu,\epsilon,h)\Sigma_{11}=\frac{\epsilon}{4}\mathbf{1}_{\{\theta=1/2\}}.$}
	
	By \eqref{avh},
	$a_{12}b_2+b_1(a_{11}-\lambda_2)=\frac{(\sqrt{\nu^2-4}-\nu)h-2\theta h^2}{2(1+\triangle)^2},$
	which together with \eqref{1-lam} and \eqref{1+lam} yields that
	\begin{align*}
	\Psi_{111}:=&\frac{a(\nu,\epsilon,h)}{(1-\lambda_1)(1+\lambda_1)}\Big(a_{12}b_2+b_1(a_{11}-\lambda_2)\Big)^2\\
	=&\frac{\epsilon\nu}{2(\nu^2-4)h}\cdot\frac{(\sqrt{\nu^2-4}-\nu)h-2\theta h^2}{(\nu-\sqrt{\nu^2-4})h+2\theta h^2}\cdot\frac{(\sqrt{\nu^2-4}-\nu)h-2\theta h^2}{
		4+\nu h(4\theta-1)+2h^2\theta(2\theta-1)+\sqrt{\nu^2-4}h}.
	\end{align*}
	By using \eqref{apprnu}, we have
	\begin{equation*} 
	\lim_{\nu\rightarrow\infty}\Psi_{111}=
	\lim_{\nu\rightarrow\infty}\frac{\nu\epsilon}{2\nu^2 h}\cdot (-1)\cdot\frac{-2\theta h^2}{\nu h(4\theta-1)+\nu h}=0. \\
	\end{equation*}
	Similarly, by combining 
	$a_{12}b_2+b_1(a_{11}-\lambda_1)=\frac{-(\sqrt{\nu^2-4}+\nu)h-2\theta h^2}{2(1+\triangle)^2},$
	\eqref{1-lam} and \eqref{1+lam}, we obtain
	\begin{align*}
	\Psi_{112}:=&\frac{a(\nu,\epsilon,h)}{(1-\lambda_2)(1+\lambda_2)}\Big(a_{12}b_2+b_1(a_{11}-\lambda_1)\Big)^2\\
	=&\frac{\epsilon\nu}{2(\nu^2-4)h}\cdot\frac{-(\sqrt{\nu^2-4}+\nu)h-2\theta h^2}{(\nu+\sqrt{\nu^2-4})h+2\theta h^2}\cdot\frac{-(\sqrt{\nu^2-4}+\nu)h-2\theta h^2}{
		4+\nu h(4\theta-1)+2h^2\theta(2\theta-1)-\sqrt{\nu^2-4}h}.
	\end{align*}
	Therefore, by \eqref{apprnu},
	\begin{equation*} 
	\lim_{\nu\rightarrow\infty}\Psi_{112}=\left\{
	\begin{split}
	&\lim_{\nu\rightarrow\infty}\frac{\nu\epsilon}{2\nu^2 h}\cdot (-1)\cdot\frac{-2\nu h}{4}=\frac{\epsilon}{4},&&\text{if}\quad\theta=1/2, \\
	&\lim_{\nu\rightarrow\infty}\frac{\nu\epsilon}{2\nu^2 h}\cdot (-1)\cdot\frac{-2\nu h}{\nu h(4\theta-1)-\nu h}=0,&&\text{if}\quad\theta\in(1/2,1]. \\
	\end{split}
	\right.
	\end{equation*}
	In addition,
	%
	by the expression of $det(A_\theta)$, we have
	\begin{align}\label{1-lam12}
	&1-\lambda_1\lambda_2=1-det(A_\theta)=\frac{(2\theta-1)h^2+\nu h}{1+\triangle},
	\end{align}
	which together with 
	$2a_{12}a_{21}b_1^2-2a_{12}^2b_2^2+2a_{12}b_1b_2(a_{22}-a_{11})=\frac{-2h^2}{(1+\triangle)^3}$
	gives that
	\begin{align*}
	\Psi_{113}:=\frac{a(\nu,\epsilon,h)}{1-\lambda_1\lambda_2}\Big(2a_{12}a_{21}b_1^2-2a_{12}^2b_2^2+2a_{12}b_1b_2(a_{22}-a_{11})\Big)=\frac{\epsilon\nu}{2(\nu^2-4)h}\cdot\frac{-2h}{(2\theta-1)h+\nu}.
	\end{align*}
	Accordingly,  for $\theta\in[1/2,1]$,
	\begin{align*}
	&\lim_{\nu\rightarrow\infty}\Psi_{113}=\lim_{\nu\rightarrow\infty}\frac{\nu\epsilon}{2\nu^2 h}\cdot\frac{-2 h}{(2\theta-1)h+\nu}=0.
	\end{align*}
	Combining the above estimates and \eqref{Sigma11}, we obtain 
	\begin{equation*}
	\lim_{\nu\rightarrow\infty}a(\nu,\epsilon,h)\Sigma_{11}=\lim_{\nu\rightarrow\infty}\Psi_{111}+\lim_{\nu\rightarrow\infty}\Psi_{112}+\lim_{\nu\rightarrow\infty}\Psi_{113}=\frac{\epsilon}{4}\mathbf{1}_{\{\theta=1/2\}}.
	\end{equation*}

	\textit{Step 2: We prove that 
		$\lim_{\nu\rightarrow\infty}a(\nu,\epsilon,h)\Sigma_{22}=\frac{\epsilon}{4}.$}
	
	In fact,  from the identity
	$a_{21}b_1+b_2(\lambda_1-a_{11})=\frac{\big(\sqrt{\nu^2-4}+\nu\big)\theta h^2+2h}{2(1+\triangle)^2},$
	it follows that
	\begin{align*}
	\Psi_{221}:=&\frac{a(\nu,\epsilon,h)}{(1-\lambda_1)(1+\lambda_1)}\Big(a_{21}b_1+b_2(\lambda_1-a_{11})\Big)^2\\
	=&\frac{\epsilon\nu}{2(\nu^2-4)h}\cdot\frac{\big(\sqrt{\nu^2-4}+\nu\big)\theta h^2+2h}{(\nu-\sqrt{\nu^2-4})h+2\theta h^2}\cdot\frac{\big(\sqrt{\nu^2-4}+\nu\big)\theta h^2+2h}{
		4+\nu h(4\theta-1)+2h^2\theta(2\theta-1)+\sqrt{\nu^2-4}h}.
	\end{align*}
	Hence, by \eqref{apprnu},
	\begin{equation*}
	\lim_{\nu\rightarrow\infty}\Psi_{221}=
	\lim_{\nu\rightarrow\infty}\frac{\nu\epsilon}{2\nu^2 h}\cdot \frac{2\nu\theta h^2}{2\theta h^2}\cdot\frac{2\nu\theta h^2}{4\nu\theta h}=\frac{\epsilon}{4}.
	\end{equation*}
	Similarly, since
	$a_{21}b_1+b_2(\lambda_2-a_{11})=\frac{(\nu-\sqrt{\nu^2-4})\theta h^2+2h}{2(1+\triangle)^2},$
	we arrive at
	\begin{align*}
	\Psi_{222}:=&\frac{a(\nu,\epsilon,h)}{(1-\lambda_2)(1+\lambda_2)}\Big(a_{21}b_1+b_2(\lambda_2-a_{11})\Big)^2\\
	=&\frac{\epsilon\nu}{2(\nu^2-4)h}\cdot\frac{(\nu-\sqrt{\nu^2-4})\theta h^2+2h}{(\nu+\sqrt{\nu^2-4})h+2\theta h^2}\cdot\frac{(\nu-\sqrt{\nu^2-4})\theta h^2+2h}{
		4+\nu h(4\theta-1)+2h^2\theta(2\theta-1)-\sqrt{\nu^2-4}h},
	\end{align*}
	which together with \eqref{apprnu} implies that
	\begin{equation*}
	\lim_{\nu\rightarrow\infty}\Psi_{222}=\left\{
	\begin{split}
	&\lim_{\nu\rightarrow\infty}\frac{\nu\epsilon}{2\nu^2 h}\cdot \frac{2h}{2\nu h}\cdot\frac{2h}{4}=0,&&\text{if}\quad\theta=1/2, \\
	&\lim_{\nu\rightarrow\infty}\frac{\nu\epsilon}{2\nu^2 h}\cdot \frac{2h}{2\nu h}\cdot\frac{2h}{\nu h(4\theta-2)}=0,&&\text{if}\quad\theta\in(1/2,1]. 
	\end{split}
	\right.
	\end{equation*}
	By noticing that for stochastic-$\theta$ scheme \eqref{stheta}, $a_{21}=-a_{12}$, we have
	\begin{align*}
	\Psi_{223}:=\frac{a(\nu,\epsilon,h)}{1-\lambda_1\lambda_2}\Big(2a_{12}a_{21}b_2^2-2a_{21}^2b_1^2-2a_{21}b_1b_2(a_{22}-a_{11})\Big)
	=\Psi_{113},
	\end{align*}
	and thereby 
	$\lim_{\nu\rightarrow\infty}\Psi_{223}=0.$ This together with \eqref{Sigma22} completes the proof of \textit{Step 2}.
	
	\textit{Step 3: We prove that 
		$\lim_{\nu\rightarrow\infty}a(\nu,\epsilon,h)\Sigma_{12}=
		0.$}
		
	Recall the definitions of $S_i,\,i=1,2,3$ given by \eqref{S1}-\eqref{S3}. Then a direct calculation gives that
	\begin{align*}
	&S_1=\frac{(\sqrt{\nu^2-4}-\nu)h^2-4\theta h^3-\theta^2(\sqrt{\nu^2-4}+\nu)h^4}{2(1+\triangle)^4},\\
	&S_2=\frac{-(\sqrt{\nu^2-4}+\nu)h^2-4\theta h^3-\theta^2(\nu-\sqrt{\nu^2-4})h^4}{2(1+\triangle)^4},\\
	&S_3=\frac{\nu h}{1+\triangle}\left\{\frac{h+h^3\theta^2}{(1+\triangle)^3}+\frac{\nu\theta h^2}{(1+\triangle)^3}\right\}=\frac{\nu h^2}{(1+\triangle)^3}.
	\end{align*}
	where we used $\eqref{avh}$  in the first two equalities,  and the facts that $a_{22}-a_{11}=\frac{\nu h}{1+\triangle},~a_{21}b_1^2-a_{12}b_2^2=\frac{h+h^3\theta^2}{(1+\triangle)^3}$
	in the last equality. Further, by \eqref{Sigma12}, we rewrite $2a(\nu,\epsilon,h)\Sigma_{12}=\Psi_{121}+\Psi_{122}+\Psi_{123}$ with
	\begin{align*}
	\Psi_{121}:=\frac{2a(\nu,\epsilon,h)}{1-\lambda_1^2}S_1,~
	\Psi_{122}:=\frac{2a(\nu,\epsilon,h)}{1-\lambda_2^2}S_2,~
	\Psi_{123}:=\frac{2a(\nu,\epsilon,h)}{1-\lambda_1\lambda_2}S_3.
	\end{align*}
	By applying \eqref{1-lam}, \eqref{1+lam} and \eqref{1-lam12}, we derive that
	
	\begin{align*}
	&\Psi_{121}
	=\frac{2\epsilon\nu}{(\nu^2-4)h}\cdot\frac{(\sqrt{\nu^2-4}-\nu)h^2-4\theta h^3-\theta^2(\sqrt{\nu^2-4}+\nu)h^4}{\Big((\nu-\sqrt{\nu^2-4})h+2\theta h^2\Big)\times \Big(4+\nu h(4\theta-1)+2h^2\theta(2\theta-1)+\sqrt{\nu^2-4}h\Big)},\\
	&\Psi_{122}
	=\frac{2\epsilon\nu}{(\nu^2-4)h}\cdot\frac{-(\sqrt{\nu^2-4}+\nu)h^2-4\theta h^3-\theta^2(\nu-\sqrt{\nu^2-4})h^4}{\Big((\nu+\sqrt{\nu^2-4})h+2\theta h^2\Big)\times \Big(4+\nu h(4\theta-1)+2h^2\theta(2\theta-1)-\sqrt{\nu^2-4}h\Big)},\\
	&\Psi_{123}
	=\frac{\epsilon\nu}{(\nu^2-4)h}\cdot\frac{\nu h^2}{(2\theta-1)h^2+\nu h}.
	\end{align*}
	Putting $\nu\rightarrow \infty$ in the above formulas and applying \eqref{apprnu} yield that 
\begin{align*}
	&\lim_{\nu\rightarrow\infty}\Psi_{121}=\lim_{\nu\rightarrow\infty}\frac{2\nu\epsilon}{\nu^2 h}\cdot \frac{-2\nu\theta^2h^4}{2\theta h^2\cdot4\nu \theta h}=0.\\
	&\lim_{\nu\rightarrow\infty}\Psi_{122}=
	\begin{cases}
	\lim_{\nu\rightarrow\infty}\frac{2\nu\epsilon}{\nu^2 h}\cdot \frac{-2\nu h^2}{2\nu h\cdot 4}=0,\qquad\qquad\,\,\text{if}\quad\theta=1/2, \\
	\lim_{\nu\rightarrow\infty}\frac{2\nu\epsilon}{\nu^2 h}\cdot \frac{-2\nu h^2}{
		2\nu h\cdot\nu h(4\theta-2)}=0,\qquad\text{if}\quad\theta\in(1/2,1]. 
	\end{cases}\\
	&\lim_{\nu\rightarrow\infty}\Psi_{123}=\lim_{\nu\rightarrow\infty}\frac{\epsilon\nu}{\nu^2 h}\cdot\frac{\nu h^2}{\nu h}=0.
	\end{align*}
	This completes the proof of \textit{Step 3}.

	Substituting the assertions of \textit{Step 1}-\textit{Step 3} into \eqref{Yhnu} produces 
	\begin{equation*}
	\Lambda^h(y)=\left\{
	\begin{split}
	&\frac{\epsilon}{4}y_1^2+\frac{\epsilon}{4}y_2^2,&&\text{if}\quad\theta=1/2, \\
	&\frac{\epsilon}{4}y_2^2,&&\text{if}\quad\theta\in(1/2,1]. 
	\end{split}
	\right.
	\end{equation*}
	Similar to \textit{Step 2} in the proof of Theorem \ref{tho4.4}, the exponential tightness of $\{\mu_{\nu,\epsilon}^h\}_{\nu>0}$ follows from the finiteness of $\Lambda^h$.
	It is apparently that $\Lambda^h$ is finite valued and Gateaux differentiable. Hence, we apply Theorem \ref{GE} to concluding that 
	for any $h>0$,
	$\{\mu_{\nu,\epsilon}^h\}_{\nu>0}$ satisfies an LDP with a good rate function. Moreover, for $\theta=1/2$, the corresponding rate function is 
	$$J^h(x)=(p^2+q^2)/\epsilon, $$
	which is exactly the rate function of $\{\mu_{\nu,\epsilon}\}_{\nu>0}$,
	and for $\theta\neq1/2$,  the corresponding rate function is 
	\begin{equation*}
	J^h(x)=\left\{
	\begin{split}
	&q^2/\epsilon,&&\text{if}\quad p=0, \\
	&\infty,&&\text{if}\quad p\neq0, 
	\end{split}
	\right.
	\end{equation*}
	where $x=(p,q)^\top\in\mbb R^2$. This proof is finished.
\end{proof}

\section{Conclusions and future work}\label{Sec6}
In this paper, the relationship between  the LDPs of invariant measure for Langevin equation \eqref{Lan} and the LDPs of invariant measures of its numerical approximations is studied. Aiming at the case that $V$ is a confining potential growing faster than some power function  at infinity, i.e., \textbf{Assumption 1} holds, we prove that the unique invariant measure $\{\mu_{\nu,\epsilon}\}$ of exact solution satisfies two kinds of LDPs in the strong dissipation limit and the small noise limit respectively. 
Then, for the linear case with  $V(q)=\frac{1}{2}q^2$, we prove that a large class of  numerical methods can asymptotically preserve the LDP of  $\{\mu_{\nu,\epsilon}\}_{\epsilon>0}$ in the small noise limit. Finally, we show that the midpoint scheme can asymptotically preserve the LDP of  $\{\mu_{\nu,\epsilon}\}_{\nu>0}$ in the strong dissipation limit. Our results indicate that different numerical methods will show the differences in preserving the different types of LDPs for the underlying systems. 

We would like to mention that our results can be generalized to the high dimensional case. In details, let \eqref{Lan} be the $2d$-dimensional Langevin equation, $V:\mbb R^d\to \mbb R$ be a confining potential and $W=(W_1,W_2,\ldots,W_d)$ be a $d$-dimensional Wiener process with $d$ being a given positive integer. Under \textbf{Assumption 1}, analogous to the arguments in Section \ref{Sec3}, one can prove that the invariant measure of the exact solution still satisfies the two LDPs in the small noise limit and strong dissipation limit respectively. As for the LDP of invariant measures of numerical solutions for \eqref{Lan} with $V(q)=\frac{1}{2}|q|^2$, we note that \eqref{Lan} can be divided into the following $d$ subsystems 
\begin{align*}\label{subLan}
\ud \begin{pmatrix}
P^{k}(t)\\
Q^{k}(t)
\end{pmatrix}=
\begin{pmatrix}
-\nu&-1\\
1&0
\end{pmatrix}
\begin{pmatrix}
P^{k}(t)\\
Q^{k}(t)
\end{pmatrix}\ud t+\sqrt{\epsilon}
\begin{pmatrix}
1\\
0
\end{pmatrix}\ud W_k(t),\qquad k=1,2,\ldots,d.
\end{align*}  
Then, we obtain the numerical method $\{(P_n,Q_n)\}_{n\geq0}$ with its $k$th component given by
\begin{align*}
\left(\begin{array}{c}
P_{n+1}^k\\\\
Q_{n+1}^k
\end{array}\right)=
\left(
\begin{array}{cc}
a_{11}(h)&a_{12}(h)\\\\
a_{21}(h)&a_{22}(h)
\end{array}\right)
\left(\begin{array}{cc}
P_n^k\\\\
Q_n^k
\end{array}\right)
+\sqrt{\epsilon} 
\left(\begin{array}{cc}
b_1(h)\\\\
b_2(h)
\end{array}\right)\Delta W_{k,n},\qquad n=0,1,2,\ldots
\end{align*}
with $(P_0,Q_0)=(P(0),Q(0))$, where $\Delta W_{k,n}= W_k(t_{n+1})-W_k(t_n)$ with $t_n=nh$, $n=1,2,\ldots$, and  $a_{ij},\,b_i: (0,\infty)\to\mbb R$, $i,j=1,2$ are  the functions of step-size $h$ and determined by a concrete method. Using the same analyses as those in Section \ref{Sec4}, one can show the asymptotical preservation of numerical methods for the LDP of the invariant measures of the exact solution  in the small noise limit.

There are some problems which remain to be solved. For the linear case with $V(q)=\frac{1}{2}q^2$, how to derive the LDP of invariant measure of general numerical methods as $\nu\to\infty$, and whether other numerical method except the midpoint scheme can asymptotically preserve the LDP of $\{\mu_{\nu,\epsilon}\}_{\nu>0}$ as $\nu\to\infty$. For the general confining potential $V$, how to construct numerical methods which possess a unique invariant measure and further derive the LDP of their invariant measures. These problems will be studied in our future work.

\bibliographystyle{plain}
\bibliography{mybibfile}

\end{document}